\documentclass[12pt]{amsart}
\usepackage[latin1]{inputenc}
\usepackage{color}
\usepackage{bbm}
\usepackage{amsmath,amssymb}
\usepackage{enumerate}
\usepackage{mathrsfs}
\usepackage{verbatim}
\usepackage{tikz}
\usepackage{ifthen}
\usetikzlibrary{calc}
\usepackage[top=2cm,bottom=2cm,left=3cm,right=3cm]{geometry}

\newtheorem{theo}{Theorem}[section]
\newtheorem{lem}[theo]{Lemma}
\newtheorem{rem}[theo]{Remark}
\newtheorem{co}[theo]{Corollary}
\newtheorem{prop}[theo]{Proposition}
\newtheorem{fact}[theo]{Fact}

\newtheorem{defn}[theo]{Definition}

\numberwithin{equation}{section}

\newcommand{\cc}{{\mathcal C}}

\newcommand{\caa}{{\mathscr A}}

\newcommand{\crr}{{\mathcal R}}
\newcommand{\cxx}{{\mathcal X}}

\newcommand{\cpp}{{\mathcal P}}
\newcommand{\cf}{{\mathscr F}}
\newcommand{\cgg}{{\mathcal G}}

\newcommand{\cw}{{\mathcal W}}

\newcommand{\rr}{\mathbb{R}}
\newcommand{\nn}{\mathbb{N}}

\newcommand{\charfun}{\ensuremath{\mathbbm 1}} 
\newcommand{\app}{{\mathcal A}}

\newcommand{\spp}{{\mathcal X}}
\newcommand{\dic}{{\mathcal D}}
\newcommand{\sm}[1]{\ensuremath{#1'}}  
\newcommand{\la}[1]{\ensuremath{#1''}} 

\newcommand{\spy}{{\mathcal Y}}

\newcommand{\vk}{{\underline{k}}}
\newcommand{\vz}{{\underline{z}}}

\DeclareMathOperator{\card}{card}
\DeclareMathOperator{\Span}{span}

\title[Properties of local orthonormal systems: variation spaces]{Properties of local orthonormal systems, \\ Part III: Variation spaces}
\author[J. Gulgowski]{Jacek Gulgowski}
\address{Faculty of Mathematics, Physics and Informatics, University of Gda\'nsk,
ul. Wita Stwosza 57, 80-308 Gda\'nsk , Poland}
\email{jacek.gulgowski@ug.edu.pl}

\author[A. Kamont]{Anna Kamont}
\address{Institute of Mathematics, Polish Academy of Sciences, Branch in Gda\'nsk, ul. Abrahama 18, 81-825 Sopot, Poland}
\email{Anna.Kamont@impan.pl}

\author[M. Passenbrunner]{Markus Passenbrunner}
\address{Institute of Analysis, Johannes Kepler University Linz, Austria, 4040 Linz, 
Altenberger Stra\ss e 69}
\email{markus.passenbrunner@jku.at, markus.passenbrunner@gmail.com}

\date{\today}
\keywords{variation spaces, approximation spaces, local orthonormal systems, greedy approximation}

\subjclass{41A46, 26A45, 42C05, 42C40, 46E30}

\begin{document}

\begin{abstract}
In  [Y.~K.~Hu, K.~A.~Kopotun, X.~M.~Yu, Constr. Approx. 2000],
the authors have obtained a characterization of best $n$-term piecewise 
polynomial approximation spaces 
as real interpolation spaces between $L^p$ and some spaces 
of  bounded dyadic ring variation.
We extend this characterization 
to the general setting of binary filtrations and 
finite-dimensional subspaces of $L^\infty$ 
as discussed in our earlier papers
 [J.~Gulgowski, A.~Kamont, M.~Passenbrunner, 
 arXiv:2303.16470 and arXiv:2304.05647].
 Furthermore, we study some analytical properties  
 of thus obtained abstract spaces of 
 bounded ring variation, as well as their connection 
 to greedy approximation by corresponding local 
 orthonormal systems.

\end{abstract}

\maketitle

\section{Introduction}

This paper is a complement to our recent papers \cite{part1, part2}, which in turn have been motivated by 
P. Petrushev \cite{pp.2003.a}. It follows by combining Theorems 3.3 and 5.3 of P.~Petrushev \cite{pp.2003.a}   that best $n$-term approximation spaces  with respect to  the $d$-variate Haar system and with respect to functions which are piecewise constant on $d$-dimensional dyadic cubes  are the same. In \cite{part1, part2} we have studied the question whether this result extends to the following more general setting. 

Let $(\Omega,\mathscr F,\mathbb P) = (\Omega,\mathscr F,|\cdot|)$ be a probability space and let $(\mathscr F_n)_{n=0}^\infty$ be 
a \emph{binary} filtration, meaning that exactly one atom of $\mathscr F_{n-1}$ is 
divided into \emph{two} atoms of $\mathscr F_n$, without any restriction 
on their respective measures. Let $S \subset L^\infty(\Omega)$ be a finite-dimensional
 linear subspace, satisfying an additional stability property on the atoms 
 $\caa$ of $\cf$. For these data, we consider two dictionaries:
\begin{itemize}
\item $\cc = \{ f \cdot \charfun_A: f \in S, A \in \caa\}$,
\item $\Phi $ -- a local orthonormal system generated by $S$ and the filtration $(\mathscr F_n)_{n=0}^\infty$.
\end{itemize}
Denoting $L^p(S) = \overline{\Span}_{L^p(\Omega)} \cc = \overline{\Span}_{L^p(\Omega)} \Phi$, 
$1 < p < \infty$, we consider the two scales of best
 $n$-term approximation spaces, $\app_q^\alpha(L^p(S), \cc)$, corresponding to the dictionary $\cc$,  and $  \app_q^\alpha(L^p(S), \Phi)$, corresponding to the dictionary $\Phi$, with $\alpha > 0 $ and $0< q \leq \infty$.
We have been interested in the following question: 
what are the conditions  on $S$ and $\cf$ such that
$\app_q^\alpha(L^p(S), \Phi) =
\app_q^\alpha(L^p(S), \cc)$?

First, in \cite{part1}, we have studied the properties 
of the orthonormal system $\Phi$ in $L^p(\Omega), 1<p<\infty$.
The main results of \cite{part1} -- Theorems 1.1 and 1.2, and their Corollary 7.3 -- state
 that for each $1 < p < \infty$, the 
 $L^p$-normalized system $\Phi_p$ resulting from $\Phi$  is a greedy basis in $L^p(S)$, 
 satisfying the $p$-Temlyakov property.
With this result at hand, we have continued with our question in \cite{part2}.
 It is clear that $\app_q^\alpha(L^p(S), \Phi) \hookrightarrow
\app_q^\alpha(L^p(S), \cc)$. 
In the other direction, it is not hard to identify a specific Bernstein 
type inequality $\operatorname{BI}(\caa, S, p, \tau)$, with parameters 
$1 < p< \infty$, $0< \tau < p$ and $\beta = 1/\tau - 1/p$, for which the following holds:
\begin{itemize}
\item The Bernstein inequality $\operatorname{BI}(\caa, S, p, \tau)$ is 
simultaneously a necessary condition for embeddings $\app_q^\alpha(L^p(S), \cc) \hookrightarrow
\app_q^\alpha(L^p(S), \Phi)$ for all $\alpha > \beta$, $0<q \leq \infty$, and a sufficient condition for embeddings 
$\app_q^\alpha(L^p(S), \cc) \hookrightarrow
\app_q^\alpha(L^p(S), \Phi)$ for all $0< \alpha < \beta$ and $0< q \leq \infty$.
\end{itemize}
The main result of \cite{part2} -- Theorem 1.1 -- is a geometric condition 
on $S$ and $\caa$, depending also on $p$ and $\tau$, 
which is equivalent to the Bernstein inequality $\operatorname{BI}(\caa, S, p, \tau)$. 
In particular, it follows that in some cases the Bernstein inequality 
$\operatorname{BI}(\caa, S, p, \tau)$ may not be satisfied, and consequently 
$\app_q^\alpha(L^p(S), \cc) \not =
\app_q^\alpha(L^p(S), \Phi)$ for all $\alpha > \beta$, $0<q \leq \infty$

Since $\Phi_p$ is a greedy basis in $L^p(S)$ with the $p$-Temlyakov property, the approximation space 
$\app_q^\alpha(L^p(S), \Phi) = \app_q^\alpha(L^p(S), \Phi_p)$
can be characterized in terms of coefficients of $f \in L^p(S)$ with respect to $\Phi_p$, see e.g. the monograph V. Temlyakov \cite{greedy}, Chapter 1.8, Theorem 1.80. 
This extends an earlier result of R. DeVore \cite[Section 7.6]{dv.1998.nl} for wavelets
to greedy bases with the $p$-Temlyakov property. 

For completeness, it seems to be of interest to have some external characterization
of spaces  $\app_q^\alpha(L^p, \cc)$, which holds also in case when
$\app_q^\alpha(L^p, \cc)  \not =  \app_q^\alpha(L^p, \Phi)$. For this, we turn to the paper Y.K. Hu, K.A. Kopotun, X.M.Yu \cite{hky.2000}. 
In \cite{hky.2000},  the authors study best $n$-term approximation spaces with respect to dictionaries $\cpp_{d,r}$ consisting of  $d$-variate polynomials of degree $r-1$ restricted to $d$-dimensional dyadic cubes. 
For this, they introduce some version of \emph{bounded ring variation} spaces $V_{\tau,p}^r[0,1]^d$ on $[0,1]^d$.
One of the main results of \cite{hky.2000} is the identification of 
the best $n$-term approximation spaces  $\app^\alpha_q(L^p[0,1]^d, \cpp_{d,r}) $ as real interpolation spaces between
$L^p[0,1]^d$ and $V_{\tau,p}^r[0,1]^d$. More precisely, by combining Lemma 3 and Corollary 8 of \cite{hky.2000},
the following is true:
\begin{itemize}
\item[(A)] We have  $\app^\alpha_q(L^p[0,1]^d, \cpp_{d,r}) = (L^p[0,1]^d, V_{\tau,p}^r[0,1]^d)_{\alpha/\beta, q}$
for $0 <\tau < p < \infty, 0 < \alpha < \beta = 1/\tau - 1/p$ and $0<q\leq\infty$.
\end{itemize}
It occurs that both the definition of the bounded ring variation space  $V_{\tau,p}^r$ 
and the  interpolation result (A) from \cite{hky.2000} can be extended to   
our more general setting, cf. Definition \ref{def.var}  and  
Section \ref{ssq.1}, with Theorem \ref{theo.101} being the main result in that direction. 
This is the first part of results of this paper.

Theorem \ref{theo.101} motivates us to look more carefully at the extended versions of 
variation spaces $V_{\sigma,p}$, 
comparing them to the \emph{usual} bounded variation space $\operatorname{BV}(\rr^d)$, as discussed e.g in W. P. Ziemer \cite{z.1989}. 
In particular,  the elements of the  space $\operatorname{BV}(\rr^d)$  can be  characterized by approximation  with smooth functions in the following sense, 
cf. Theorems 5.2.1 and 5.3.3
of \cite{z.1989}:
\begin{itemize}
\item[(B)]
$f \in \operatorname{BV}(\rr^d)$ if and only if there is a sequence of $f_n \in W_1^1(\rr^d)$ such that 
\begin{equation}\label{int.e1}
 \| f - f_n\|_1 \to 0 \quad \hbox{and}\quad \sup_n | f_n |_{W_1^1} < \infty.
\end{equation}
Moreover, 
$$
| f |_{\operatorname{BV}} = \inf_{(f_n)} \liminf_{n \to \infty} | f_n |_{W_1^1},
$$
where the infimum is taken over all sequences $(f_n)$ satisfying \eqref{int.e1}.
\end{itemize}
We are interested in obtaining a variant of (B) for spaces $V_{\sigma,p}$, with finite linear combinations of elements of $\cc$ replacing smooth functions.
This is the content of  Section \ref{ssq.2}, with Theorem \ref{theo.txt.1} 
and Corollary~\ref{co:liminf} being 
the main results in this direction. 
This is the second part of our results. 
Note that Theorem \ref{theo.txt.1} can be seen as a variant 
of Theorem 3 in P. Wojtaszczyk \cite{pw.2003} for $\operatorname{BV}(\rr^d)$, or its version for averaging projections with respect to dyadic cubes in $\rr^d$, cf. page  488 of \cite{pw.2003}, for spaces $V_{\sigma,p}$ instead of 
$\operatorname{BV}(\rr^d)$.
Let us recall that Yu. Brudnyi \cite{yb.2017} 
has used another version of bounded variation 
spaces for his study of best $n$-term approximation 
by polynomials on dyadic cubes.
Theorem 3.2 of \cite{yb.2017} can be regarded as a counterpart
 of (B) in the setting of \cite{yb.2017}.

In the first two sets of results we are able to deal with the range $0<p<\infty$, similarly as in Y.K. Hu, K.A. Kopotun, X.M.Yu \cite{hky.2000}.

The last problem treated in this paper is the question whether the K-functional for $f \in L^p(S)$ and the pair $(L^p(S),V_{\sigma,p})$ is attained 
for greedy approximation of $f$ with respect to $\Phi_p$ in $L^p(S)$. This question is motivated by a positive answer to the analogous question for the pair 
$(L^{d^*}(\rr^d), \operatorname{BV}(\rr^d))$ with $d^* = \frac{d}{d-1}$  and greedy approximation of $f$ with respect to the $d^*$-normalized Haar system, or
a $d^*$-normalized
 compactly supported wavelet system,   cf. A. Cohen, R. Devore, P. Petrushev, H. Xu \cite{cdpx.1999}, Theorem 9.2 ($d=2$, Haar system),
P. Wojtaszczyk \cite{pw.2003}, Theorem 12 ($d \geq 2$, Haar system), or P. Bechler, R. DeVore, A. Kamont, G. Petrova, P. Wojtaszczyk \cite{bv.2007},
Theorem 1.2 ($d\geq 2$,  compactly supported wavelet system). More precisely:
\begin{itemize}
\item[(C)] We have the equivalence 
 $$K(f, L^{d^*}(\rr^d), \operatorname{BV}(\rr^d), N^{-1/d}) \simeq \| f - \cgg_N f \|_{d^*} + N^{-1/d} |\cgg_N f |_{\operatorname{BV}},$$
with $\cgg_N f$ being $N$-term greedy approximation in $L^{d^*}(\rr^d)$ to $f$ with respect to a  $L^{d^*}$-normalized
compactly supported and sufficiently regular wavelet system; in particular, the Haar system is allowed.
\end{itemize}

Because of the role of the Bernstein-type inequality $\operatorname{BI}(\caa, S, p, \tau)$
in the equality $\app_q^\alpha(L^p(S), \cc) =  \app_q^\alpha(L^p(S), \Phi)$, 
it is natural to consider this question under the assumption that 
the Bernstein-type inequality $\operatorname{BI}(\caa, S, p, \tau)$ holds true. 
In fact, Proposition~\ref{prop:notw3} in Section \ref{sec:k-functional} indicates 
 that an analogue of the equivalence (C) requires some geometric interplay between 
 $S$ and $\caa$. Unfortunately, we have not been able to answer this question in 
 this generality. However, we are able to prove some embeddings between the space
 $V_{\sigma,p}$ and the spaces $\app_q^\alpha(L^p(S),\Phi)$, where
  $\alpha=1/\sigma -1/p$, cf. Theorems \ref{thm:embed_app_var} and \ref{thm:embed_var_app}. 
These embeddings imply some weak version of the positive answer to the question on the $K$-functional, cf. Theorem~\ref{thm:K_equiv}.

The paper is organized as follows. The necessary background material is recalled 
in Section~\ref{sec:def}.
The variation spaces $V_{\sigma,p}$ and their properties are discussed 
in Section~\ref{txt.s1}, with the counterpart of (A) and related results
 presented in Section \ref{ssq.1}, and the counterpart of (B) presented in 
 Section \ref{ssq.2}. The proofs of the results from Section \ref{ssq.1} 
 follow the arguments analogous to their counterparts in \cite{hky.2000}, 
 so their proofs are postponed to Appendix  \ref{app.s1}. Finally, 
 in Section \ref{sec:embedding},
we discuss embeddings between the spaces $V_{\sigma,p}$ and 
$\app_q^\alpha(L^p(S), \Phi_p)$, and we use them to get a weak version of (C).

\section{Definitions and preliminaries}
\label{sec:def}

\subsection{Approximation spaces } \label{sec:appr}
In this section, we collect what is needed in the sequel 
about approximation spaces, as 
given for instance in  \cite[Chapter 7, Sections 5 and 9]{constr.approx}.
Let $(\spp, \| \cdot \|)$ be a quasi-normed space, and let $\dic \subset \spp$ be a subset that is linearly dense in $\spp$; the set $\dic$ is called a \emph{dictionary}.
We are interested in approximation spaces corresponding to the  best approximation by $n$-term linear  combinations of elements of 
$\dic$. That is, let for $n \in \nn$
$$
\Sigma_n = \Sigma_n^\dic = \Big\{ \sum_{j=1}^n c_j x_j : \text{with }x_j \in \dic \text{ and } \text{scalars } c_j \Big\}.
$$
Then, the smallest $n$-term approximation error by linear combinations of $\dic$ is defined by
$$
 \sigma_n(x) = \sigma_n(x,\dic) =  \inf\{\| x - y\|: y \in \Sigma_n\},\qquad x\in \spp.
$$
The approximation spaces $\app_q^\alpha(\spp, \dic)$ with $0 < q \leq \infty$, $\alpha >0$, are defined by
$$
\app_q^\alpha(\spp, \dic) = \{ x \in \spp: \| x \|_{\app_q^\alpha} = \| x \| + 
 \| \{ 2^{n \alpha} \sigma_{2^n}(x), n \geq 0 \} \|_{\ell^q}  < \infty \}.
$$

We will also refer to the following fact about interpolation among those spaces, 
which is a direct consequence of 
\cite[Chapter~7, Theorem~9.1]{constr.approx}:
\begin{fact}
\label{intro.f1}
Let $0< q, \rho \leq \infty$ and $0 < \alpha < \beta $. Then $\app_q^\alpha(\spp, \dic) = (\spp, \app_\rho^\beta(\spp, \dic) )_{\alpha/\beta,q}$.
\end{fact}

\subsection{Greedy approximation}\label{sec:greedy}
If $\Psi=(\psi_j)$ is a basis of the Banach space $\mathscr \spp$ and $f\in\spp$ 
has the expansion $ f = \sum_j c_j \psi_j$, we choose a permutation $(k_j)$ of the 
positive integers such that the respective coefficients are decreasing:
\[
	|c_{k_1}| \geq |c_{k_2}| \geq \cdots.
\]
Then, we define the $m$th greedy approximant of $f$ with respect to the 
basis $\Psi$ by 
\[
	\mathscr G_m(f):= \mathscr G_m(f,\Psi) := \sum_{j=1}^m c_{k_j} \psi_{k_j}.
\]
Recall that the basis  $\Psi$ is a \emph{greedy basis} in $\spp$ 
if the smallest  $m$-term approximation error $\sigma_m(x,\Psi)$ of each  
$x\in\spp$ is achieved by greedy approximation: 
\[
\sigma_m(f,\Psi) \simeq \| f - \mathscr G_m (f,\Psi)\|_{\spp}.	
\]
A basis  $\Psi = (\psi_j)$  satisfies 
the \emph{$p$-Temlyakov property} (cf. \cite[Equation~(1.130)]{greedy}) in $\spp$, 
if
for some constant $C$, the 
inequality 
\begin{equation}\label{eq:temlyakov}
C^{-1}(\card \Lambda)^{1/p}\leq \Big\| \sum_{n\in\Lambda} \psi_n \Big\|_\spp \leq C (\card \Lambda)^{1/p}	
\end{equation}
is true for every finite index set $\Lambda$, where $\card \Lambda$ denotes the 
number of elements of $\Lambda$.

In the case of greedy bases $\Psi$ that satisfy the $p$-Temlyakov 
property, the space $\app_q^\alpha(\spp, \Psi)$ can be described in terms of 
the coefficients $(c_n)$ in the expansion $x = \sum c_n \psi_n$  
(see \cite[Theorem 1.80]{greedy}):
\begin{fact}
\label{intro.f3}
Let $\Psi=(\psi_n)$ be a greedy basis satisfying the $p$-Temlyakov property. 
Then we have for $0< \beta,q<\infty$
$$\app_q^\beta(\spp, \Psi) = \Big\{ x = \sum_n c_n \psi_n \in \spp :  
\|( 2^{j(\beta + 1/p )} a_{2^j} )_j \|_{\ell^q }<\infty  \Big\},$$
with equivalence of (quasi-)norms $\| x \|_{\app_q^\beta} \simeq  
\|( 2^{j(\beta + 1/p )} a_{2^j} )_j \|_{\ell^q }$,
where the sequence $(a_k)$ denotes the decreasing rearrangement of $(c_n)$.
\end{fact}
In particular, if we choose $\beta = 1/q - 1/p$ for $0<q<p<\infty$, we obtain 
the equivalence
\begin{equation}\label{eq:special_beta}
\|x\|_{\app_q^\beta} \simeq \|(c_n)\|_{\ell^q}.
\end{equation}

\subsection{Local orthonormal systems}\label{sec:loc}
Let $(\Omega,\mathscr F,\mathbb P) = (\Omega,\mathscr F,|\cdot|)$ be a probability space and let $(\mathscr F_n)_{n=0}^\infty$ be 
a \emph{binary} filtration, meaning that
\begin{enumerate}
	\item $\mathscr F_0 = \{\emptyset,\Omega\}$,
	\item for each $n\geq 1$, $\mathscr F_n$ is generated by $\mathscr F_{n-1}$ and 
			the subdivision of exactly one atom $A_n$ of $\mathscr F_{n-1}$ into
			two atoms $A_{n}', A_n''$ of $\mathscr F_n$ satisfying $|A_n''|\geq |A_n'|>0$.
\end{enumerate}
If $A\in\mathscr A$ is such that it strictly contains another atom from $\mathscr A$, there is a unique index
$n_0 = n_0(A)\geq 1$ so that $A = A_{n_0}\in\mathscr A_{n_0 -1}$. 
We set $\sm{A} := \sm{A_{n_0}}$ and $\la{A} := \la{A_{n_0}}$.

We are also interested in  more general $\nu$-ary
 filtrations.
For an integer $\nu \geq 2$, we mean by a \emph{$\nu$-ary} filtration $(\mathscr F_n)_{n=0}^\infty$ 
a  filtration so that for each $n\geq 1$, each atom $A$ of $\mathscr F_{n-1}$ 
can be written as the union of at most $\nu$ atoms of $\mathscr F_n$. Note that each $\nu$-ary
 filtration can be regarded as a subsequence of some binary filtration, but the choice of such a binary filtration is not unique.
 $\nu$-ary filtrations 
arise in particular for local tensor products as considered in \cite[Section~9]{part1},
where a concrete example with $\nu = 2^d$ is given by the dyadic filtration $(\mathscr D_n)$ on the 
$d$-dimensional unit cube $[0,1]^d$ where $\mathscr D_n$ is generated by the dyadic cubes of 
sidelength $2^{-n}$.

For a $\nu$-ary filtration $(\mathscr F_n)$, let $\mathscr A_n$ be the collection of all atoms of
$\mathscr F_n$ and set $\mathscr A = \cup_{n} \mathscr A_n$.
In the sequel, we refer to sets $A\in\mathscr A$ as \emph{atoms} and differences
$A\setminus B\neq \emptyset$  of two atoms $A,B$ 
 are called \emph{rings}.

Let 
$S$ be a finite-dimensional linear space of $\mathscr F$-measurable, scalar-valued functions on $\Omega$ 
so that
there exist two constants $c_1,c_2\in (0,1]$ so that for each
atom $A\in\mathscr A$ we have the following stability inequality:
		\begin{equation}\label{eq:L1Linfty}
			|\{ \omega\in A : |f(\omega)| \geq c_1\|f\|_{A}
		\}| \geq c_2|A|,\qquad f\in S,
	\end{equation}
	where by $\|f\|_A$ we denote the $L^\infty$-norm of $f$ on the set $A$.
Given any set $V$ of $\mathscr F$-measurable functions on $\Omega$
 and any measurable set $A\in\mathscr F$, we let 
\[
	V_A := V(A) := \{ f\cdot \charfun_A : f\in V\},
\]
where by $\charfun_A$ we denote the characteristic function of the set $A$ defined by 
$\charfun_A(x) = 1$ if $x\in A$ and $\charfun_A(x)=0$ otherwise.
Moreover, for any $\sigma$-algebra $\mathscr G\subset \mathscr F$, we let 
\[
	V(\mathscr G)= \{ f : \Omega\to\mathbb R\;|\;  \text{for each atom $A$ of $\mathscr G$ there exists 
$g\in V$ so that $f\charfun_A = g\charfun_A$} \}.
\]
We also use the abbreviation $V_n := V(\mathscr F_n)$.
Let $P_n$ be the orthoprojector onto $S_n$ for $n\geq 0$ and set $P_{-1}\equiv 0$.
For each integer $n\geq 0$, 
	 we choose an arbitrary orthonormal basis $\Phi_n$
	of the range of $P_n - P_{n-1}$ and define 
	the orthonormal system $\Phi = \cup_{n=0}^\infty \Phi_n$.
	The collection $\Phi = (\varphi_j)$ is called a \emph{local orthonormal system}, since 
	it is easy to see that functions in the range of $P_n  -P_{n-1}$ are supported 
	in the set $A_n$.
	
Now, assume that $(\mathscr F_n)$ is a binary filtration and the stability condition \eqref{eq:L1Linfty} is satisfied.
In \cite{part1}, we showed that for every $1<p<\infty$, the $p$-renormalized system $\Phi$ is greedy in 
$L^p$ and satisfies the $p$-Temlyakov property.
In case of a $\nu$-ary filtration satisfying \eqref{eq:L1Linfty}, 
there exists a local orthonormal system $\Phi$ such that for each $1<p<\infty$,
 the $p$-renormalized system $\Phi$ is  greedy in $L^p$, and it satisfies the $p$-Temlyakov property,
  cf. \cite[Section~9]{part1}.  There is a canonical way to get such a system 
   by viewing a $\nu$-ary filtration satisfying stability condition 
\eqref{eq:L1Linfty} as a subsequence of a binary filtration satisfying condition 
\eqref{eq:L1Linfty} (with some other constants $c_1',c_2'$). However, let us remark 
that it is possible to construct such a system $\Phi$ which is not directly 
linked to some binary filtration, the first  example being  the $d$-variate 
Haar system on the dyadic filtration on $[0,1]^d$, $d \geq 2$. More generally, 
it is possible to get such examples in case of filtrations with local tensor 
product structure. It can be noted that if a local orthonormal system is such
 that its $p$-renormalization is greedy in $L^p$, $1<p<\infty$, then it satisfies 
$p$-Templakov property, cf. \cite[Theorem 9.4]{part1} and its proof.
On the other hand, one should be aware that in case of non-binary filtrations, there may exist a local orthonormal system $\Phi$ which is not a basis in its $L^p$-span, for any $p \neq 2$.
Thus,  both here below and in Section~\ref{sec:embedding}, we make the following assumption: $\Phi$ is any local orthonormal system corresponding to the $\nu$-ary filtration  $(\mathscr F_n)$ 
and such that for each $1< p < \infty$, its $p$-renormalization is  greedy in $L^p$. 

 Therefore, all results 
of Section~\ref{sec:greedy} in particular apply to $p$-normalized local orthonormal systems $\Phi$ which are greedy in $L^p$.

We now consider the two dictionaries
$\cc = \{ f \cdot \charfun_A: f \in S, A \in \caa\}$ and 
$\Phi$ and the corresponding approximation spaces. 
Let $L^p(S) = \overline{{\rm span}}_{L^p(\Omega)}
 \cc =  \overline{{\rm span}}_{L^p(\Omega)} \Phi$. 
 Observe that $\Phi \subset \Sigma_{\nu}^{\cc}$.
As $S$ is finite-dimensional, we have the continuous embedding $\app_q^\alpha (L^p(S), \Phi)
\hookrightarrow \app_q^\alpha (L^p(S), \cc)$ for all $\alpha,q>0$.  
The answer to the question when those two approximation spaces coincide 
is governed by the
validity of the following Bernstein inequality.
\begin{defn}
    Fix $1<p<\infty$ and $0<\tau<p$ and let $\beta := 1/\tau - 1/p >0$. We say that 
    the \emph{Bernstein inequality} $\operatorname{BI}(\mathscr A,S,p,\tau)$ is satisfied 
    if there exists a constant $C$ such that for all positive integers $n$ and all $g\in \Sigma_{2^n}^{\Phi}$
    we have 
\[
    \| f \|_{\mathcal A_\tau^\beta(L^p(S), \Phi)}     \leq C 2^{n\beta} \|f\|_{L^p}.
\]
\end{defn}
Then, we recall the following theorem (see \cite[Theorem~2.6]{part2}; in \cite{part2}, it is stated for binary filtrations, but it is true for $\nu$-ary filtrations as well).
\begin{theo}\label{bernst.equiv} 
	Fix $1<p<\infty$ and $0<\tau<p$ and let $\beta := 1/\tau - 1/p>0$. Then:
	\begin{itemize}
\item[(A)]  
If the Bernstein inequality $\operatorname{BI}(\caa, S, p, \tau) $ is not satisfied, then  for all $\alpha  > \beta$, $0 < q \leq \infty$
there is $\app^\alpha_q (L^p(S), \Phi) \hookrightarrow \app^\alpha_q  (L^p(S), \cc)     $, but $\app^\alpha_q (L^p(S), \Phi) \neq \app^\alpha_q  (L^p(S), \cc)   $
\item[(B)] If the Bernstein inequality  $\operatorname{BI}(\caa, S, p, \tau) $ is satisfied, then for all $0< \alpha  < \beta$, $0 < q \leq \infty$
there is $ \app^\alpha_q(L^p(S), \cc) = \app^\alpha_q (L^p(S), \Phi) $. 
\end{itemize}
\end{theo}

The main result of \cite{part2} is that in case of a binary filtration, the Bernstein inequality $\operatorname{BI}(\mathscr A,S,p,\tau)$ is equivalent 
to the geometric condition $\operatorname{w2^*}(\mathscr A,S,p,\tau)$, defined as follows.

\begin{defn} \label{def:w2star}
The condition  
$\operatorname{w2}^*(\mathscr A,S,p,\tau)$ 
is satisfied if
there exist a number $\rho\in (0,1)$ and a constant $M$ such that 
for each finite chain $X_0 \supset X_1 \supset \cdots \supset X_n$ of atoms in $\mathscr A$
with $|X_n|\geq \rho |X_0|$ and $X_j\in \{X_{j-1}',X_{j-1}''\}$ for each $j$, we have the inequality
\begin{equation}\label{eq:introw2s}
\Big( \sum_{i=1}^n \| f \charfun_{X_{i-1}\setminus X_i} ) \|_p^\tau \Big)^{1/\tau} 
\leq M \| f \charfun_{X_0\setminus X_n} \|_p,\qquad f\in S.
\end{equation}

\end{defn}

In order to illustrate condition~$\operatorname{w2}^*$, we consider the piecewise constant 
case $S=\operatorname{span}\{\charfun_\Omega\}$ and give an explicit example 
why $\app_\tau^\beta(L^p(S),\cc)$ cannot be continuously embedded in 
$\app_\tau^\beta(L^p(S),\Phi)$ for $\beta = 1/\tau - 1/p$ 
when condition $\operatorname{w2}^*(\mathscr A,S,p,\tau)$ is not 
satisfied.

Note that in case of $S=\operatorname{span}\{\charfun_\Omega\}$, inequality \eqref{eq:introw2s} takes the form
\begin{equation}\label{eq:example}
\Big( \sum_{i=1}^n | X_{i-1}\setminus X_i |^{\tau/p} \Big)^{1/\tau} 
\leq M |X_0\setminus X_n|^{1/p} .
\end{equation}

Assume that $\operatorname{w2}^*(\mathscr A,S,p,\tau)$ is not satisfied for
some parameters $1<p<\infty$, $0<\tau<p$ and set $\beta = 1/\tau - 1/p$. Thus, for a parameter $\rho<1$ sufficiently large, 
and for every $N\in\mathbb N$,
there exists a chain $\mathscr X = (X_j)_{j=0}^n$ of atoms
with $|X_n|\geq \rho |X_0|$ 
satisfying
\begin{equation}\label{eq:not_bernstein}
 \Big(\sum_{i=0}^{n-1} |X_i'|^{\tau/p}\Big)^{1/\tau}   = \Big(\sum_{i=0}^{n-1} \| \charfun_{X_i'} \|_p^\tau\Big)^{1/\tau} \geq N \| \charfun_{X_0\setminus X_n}\|_p = N |X_0\setminus X_n|^{1/p}.
\end{equation}
In this case, the orthonormal system $\Phi=(\varphi_j)$ is a generalized Haar 
system. We denote by $(h_i)_{i=0}^{n-1}$ the subset of $\Phi$ with 
$\operatorname{supp} h_i = X_i$ for $i=0,\ldots,n-1$.
It is easy to see that the Haar functions $(h_i)$ satisfy
\begin{align*}
    |h_i| \simeq |X_i'|^{-1/2}\text{ on }X_i'\qquad\text{and}\qquad
    |h_i| \simeq |X_i'|^{1/2} /|X_i| \text{ on }X_i''.
\end{align*}
This implies $\|h_i\|_p \simeq |X_i'|^{1/p - 1/2}$. Thus, the 
$p$-normalized version $h_i^p = h_i / \|h_i\|_p$ satisfies 
\begin{align*}
    |h_i^p| \simeq |X_i'|^{-1/p}\text{ on }X_i'\qquad\text{and}\qquad
    |h_i^p| \simeq |X_i'|^{1-1/p} /|X_i| \text{ on }X_i''.
\end{align*}
By \eqref{eq:special_beta}, $\| g\|_{\mathcal A_\tau^\beta(L^p(S),\Phi)} \simeq \big(\sum_j |c_j|^\tau)^{1/\tau}$
for $g = \sum_j c_jh_j^p$.
Note that  $c_j \simeq  \langle g, h_j^{p'}\rangle$ with $p' = p/(p-1)$ being the dual exponent 
to $p$.
The above estimates imply $| \int_{X_i'} h_i^{p'} | \simeq |X_i'|^{1/p}$ and, by orthogonality 
also  $| \int_{X_i''} h_i^{p'} | \simeq |X_i'|^{1/p}$.
Letting $Y= X_i'' \setminus X_n$ (and thus $Y\cup X_i' = X_i\setminus X_n$), we get
\[
\Big| \int_Y h_i^{p'} \Big| \lesssim |Y| \frac{|X_i'|^{1/p}}{|X_i|}.
\]
Since we assumed $\rho$ sufficiently large, we obtain $| \int_{X_0\setminus X_n} h_i^{p'} |
\simeq |X_i'|^{1/p}$.

Now, we determine the norm of the function $f := \charfun_{X_0\setminus X_n}$
in the  approximation spaces $\app_\tau^\beta$ with the two dictionaries $\cc$ 
and $\Phi$. Note first that $f$ is the difference of two functions in the dictionary $\cc$.
Therefore, $\| f\|_{\app_\tau^\beta(L^p(S),\cc)} \simeq \|f\|_p = |X_0 \setminus X_n|^{1/p}$.
Next, we get by the above remarks
\[
    \| f \|_{\app_\tau^\beta(L^p(S),\Phi)} \simeq \Big(\sum_{i=0}^{n-1} |\langle f,h_i^{p'}\rangle|^\tau\Big)^{1/\tau}
    \simeq \Big(\sum_{i=0}^{n-1} |X_i'|^{\tau/p}\Big)^{1/\tau}.
\]
Thus, by assumption \eqref{eq:not_bernstein}, we obtain 
\[
\| f \|_{\app_\tau^\beta(L^p(S),\Phi)} \gtrsim N 
 \| f \|_{\app_\tau^\beta(L^p(S),\cc)},
\]
for all $N\in\mathbb N$ with a uniform constant. This implies that 
$\mathcal A_{\tau}^\beta(L^p(S),\cc)$ cannot
be continuously embedded in $\mathcal A_{\tau}^\beta(L^p(S),\Phi)$ if we
assume that condition $\operatorname{w2}^*(S,\mathscr A,p,\tau)$ is not satisfied.

\section{Variation spaces and their properties \label{txt.s1}} 
We continue to use the notation introduced in Section~\ref{sec:def}
and assume we are given a $\nu$-ary filtration $(\mathscr F_n)$
for some integer $\nu\geq 2$.

\subsection{Definitions}

Let us now fix $f\in L^p(\Omega)$  and $R\in\cf$. Then we define
\[
E_p(f,R) = \inf\{ \Vert f-g\Vert_{L^p(R)} : g\in S_R\}.
\]

Now, we make the following Definition \ref{def.var}, 
which is an abstract version of the space of functions of bounded (ring)
 variation in \cite[Section 2]{hky.2000}:

\begin{defn}
\label{def.var}
Let $0<\sigma<p<\infty$ and let $f\in L^p(S)$. Then
\[
|f|_{V_{\sigma,p} } := \sup_\Pi \left( \sum_{R\in\Pi} E_p(f,R)^\sigma \right)^{1/\sigma}.
\]
where
$\sup$ is taken over all finite collections $\Pi$ of disjoint atoms or rings.
 Additionally
\[
V_{\sigma,p} := V_{\sigma,p}(S) := \{ f\in L^p(S) : \|f\|_{V_{\sigma,p}} := \| f\|_p + |f|_{V_{\sigma,p}(S) } < \infty \}.
\]
\end{defn}

\subsection{\label{ssq.1} Interpolation results}
Our interest in the space $V_{\sigma,p}(S)$ comes from Theorem~\ref{theo.101}, which is 
a generalization of  \cite[Corollary 8]{hky.2000}.

\begin{theo}
\label{theo.101}
Let $0< \sigma < p < \infty$ and $\beta = 1/\sigma - 1/p$. Then for all $0< \alpha < \beta$ and $0 < q \leq \infty$,
$$
  \app^\alpha_q(L^p(S), \cc) = (L^p(S), V_{\sigma,p}(S))_{\alpha/\beta,q}.
$$
\end{theo}

The proof of Theorem \ref{theo.101} relies on the following Jackson and Bernstein inequalities, which should be seen as a generalization of \cite[Corollary 7]{hky.2000}:

\begin{prop}\label{th.kfun.c} For $0 < \sigma < p < \infty$ and $\beta = 1/\sigma - 1/p$, 
  the following results are true.

(i) [Jackson inequality] There exists a constant $c$ such that  
 \[
 \sigma_n(f,\cc) \leq \frac{c}{n^\beta} |f|_{V_{\sigma,p}}, \qquad f\in V_{\sigma,p}(S),
 \]
 where $\sigma_n(f,\cc)$ is taken with respect to $L^p$-norm.

(ii) [Bernstein inequality]  There exists constant $c>0$ such that 
 \begin{equation}
 |f|_{V_{\sigma,p}(S)} \leq c\cdot n^\beta \Vert f\Vert_p,\qquad f\in\Sigma_n^{\cc}.
 \end{equation}
\end{prop}

Theorem \ref{theo.101} follows by Proposition \ref{th.kfun.c}  via a standard argument, see \cite[Chapter 7, Theorem 9.1]{constr.approx}.
We give the proofs of  Proposition \ref{th.kfun.c} and Theorem~\ref{theo.101}  in  Appendix \ref{app.0.ss1}.

An additional tool in the proof of Proposition~\ref{th.kfun.c} is the following lemma, 
which is a variant -- or a special case -- of Lemma 13 in \cite{hky.2000}.
The  proof of Lemma \ref{lem.bernst} follows the lines of the proof of  Lemma 13 in \cite{hky.2000}
and the details can be found in Appendix~\ref{app.0.ss1} as well.

\begin{lem}\label{lem.bernst}Let $f\in S_A$, $A\in{\mathscr A}$. Let ${\mathcal P}$ be a family of pairwise disjoint atoms or rings. Then
\begin{equation*}
\card\{ R\in{\mathcal P} : E_p(f,R) \neq 0 \} \leq 1.
\end{equation*}
\end{lem}

Next we give an equivalent form of the $K$-functional for the interpolation pair appearing in Theorem \ref{theo.101}.
We discuss such a form and thereby follow paper \cite{hky.2000}.

The following Definition \ref{def.108} should be compared to an analogous definition in \cite{hky.2000}, cf. page 453 of \cite{hky.2000}:
\begin{defn}\label{def.108}
  For $0<\sigma<p<\infty$, let the modulus of smoothness be given by
\begin{equation*}
\cw_S(f,t)_{\sigma,p} = \sup_{\Pi} \min\Big(t,\frac{1}{\card\Pi}\Big)^\beta \left( \sum_{R\in\Pi} E_p(f,R)^\sigma\right)^{1/\sigma},
\end{equation*}
where $\beta = 1/\sigma - 1/p$ and the supremum is taken over all finite partitions $\Pi$ of $\Omega$ into disjoint atoms or rings.
\end{defn}

This modulus of smoothness is related to Theorem \ref{theo.101} via the following Theorem~\ref{theo.234}, which corresponds to  \cite[Theorem~5]{hky.2000}.

\begin{theo}\label{theo.234} For all $0<\sigma<p<\infty$, there exist constants $c_1,c_2>0$ such that for all $f\in L^p(S)$
\begin{equation*}
c_1  \cw_S(f,t)_{\sigma,p} \leq K(f,t^\beta, L^p(S), V_{\sigma,p}) \leq c_2 \cw_S(f,t)_{\sigma,p}.
\end{equation*}
\end{theo}
The details of the proof can be found in Appendix~\ref{sec:appendix_K_equiv}.

\subsection{Comparison of variation spaces corresponding to binary 
and $\nu$-ary refinements}

In this section we discuss the differences between variation spaces 
corresponding to binary and $\nu$-ary refinement for $\nu > 2$. 
It turns out that the resulting variation spaces are in general not the same. 
In order to show this, we provide an example defined on 
some self-similar fractal, equipped with the corresponding Hausdorff measure.
This example is given for any  $\nu\geq 4$.
The special geometry of this example and the corresponding function $f$ 
that are contained in one variation space, but not in the other, is chosen such that it is easy 
to count non-overlapping  rings $R$ of fixed size for which $E_p(f,R) \neq 0$.

For an example defined on the simpler set $[0,1)^2$ equipped with the 
Lebesgue measure, we refer to Appendix~\ref{sec:tensor_example}.
We remark that the example given there  works
only for $\nu = n_1 \cdot n_2$ with integers $n_1,n_2\geq 2$. 

We now present the announced
example involving self-similar fractals. For more information in that 
direction, see for instance the book \cite{Falconer1986} by K.J.~Falconer.
Fix $n \geq 2$ and choose an integer $\nu$ in the closed, non-empty interval $[n+2,n^2]$.
For $j \geq 1$, let
$$
\Gamma_j  =  \{ \vk = (k_1, \ldots, k_j): 1 \leq k_i \leq \nu\},
\quad
\Gamma_j^*  =  \{ \vk = (k_1, \ldots, k_j): 1 \leq k_i \leq n\}.
$$
Moreover, let
$$
Q = [0,1)^2, \quad \Delta = \{ (t,t): 0 \leq t < 1\},
$$
$$
\Lambda = \Big\{ \Big( \frac{i_1}{n}, \frac{i_2}{n} \Big): 0 \leq i_1, i_2 \leq n-1 \Big\}.
$$
Fix $\nu$ distinct points $\xi_0, \ldots, \xi_{\nu-1} \in \Lambda $ by the following rule: first, let $\xi_i = (i/n, i/n)$ for $i = 0, \ldots, n-1$,  
and choose $\nu-n$ more arbitrary points 
$\xi_{n},\ldots, \xi_{\nu-1} \in \Lambda $ so that the set $\{\xi_0, \ldots, \xi_{\nu-1}\}$ consists of $\nu$ points.

%
%
%

\begin{figure}
\begin{tikzpicture}[x=10cm,y=10cm]
\draw[draw=black,dashed] (0,0) rectangle (1,1);
\draw[draw=black] (0,0) -- (1,1);
\node[anchor=north west] (A) at (1/2,1/2)  {$\Delta$};
\filldraw[black] (0,0) circle (2pt) node[anchor=north west]{$\xi_0$};
\filldraw[black] (1/3,1/3) circle (2pt) node[anchor=north west]{$\xi_1$};
\filldraw[black] (2/3,2/3) circle (2pt) node[anchor=north west]{$\xi_2$};
\filldraw[black] (0,2/3) circle (2pt) node[anchor=north west]{$\xi_3$};
\filldraw[black] (2/3,1/3) circle (2pt) node[anchor=north west]{$\xi_4$};

\draw[draw=black,thick] (0,0) rectangle (1/3,1/3);
\draw[draw=black,thick] (1/3,1/3) rectangle (2/3,2/3);
\draw[draw=black,thick] (2/3,2/3) rectangle (1,1);
\draw[draw=black,thick] (2/3,1/3) rectangle (1,2/3);
\draw[draw=black,thick] (0,2/3) rectangle (1/3,1);

\node[anchor = south east] (A)  at (1/3,2/3) {$g_3(Q)$};
\node[anchor = south east] (A)  at (1,2/3) {$g_2(Q)$};
\node[anchor = south east] (A)  at (2/3,1/3) {$g_1(Q)$};
\node[anchor = south east] (A)  at (1/3,0) {$g_0(Q)$};
\node[anchor = south east] (A)  at (1,1/3) {$g_4(Q)$};

\end{tikzpicture}
\caption{The unit square $Q$, the points $(\xi_i)$, and the functions $g_i$ in the case $\nu = 5 = n+2$.  }
    \label{fig:counter2_alt}
\end{figure}
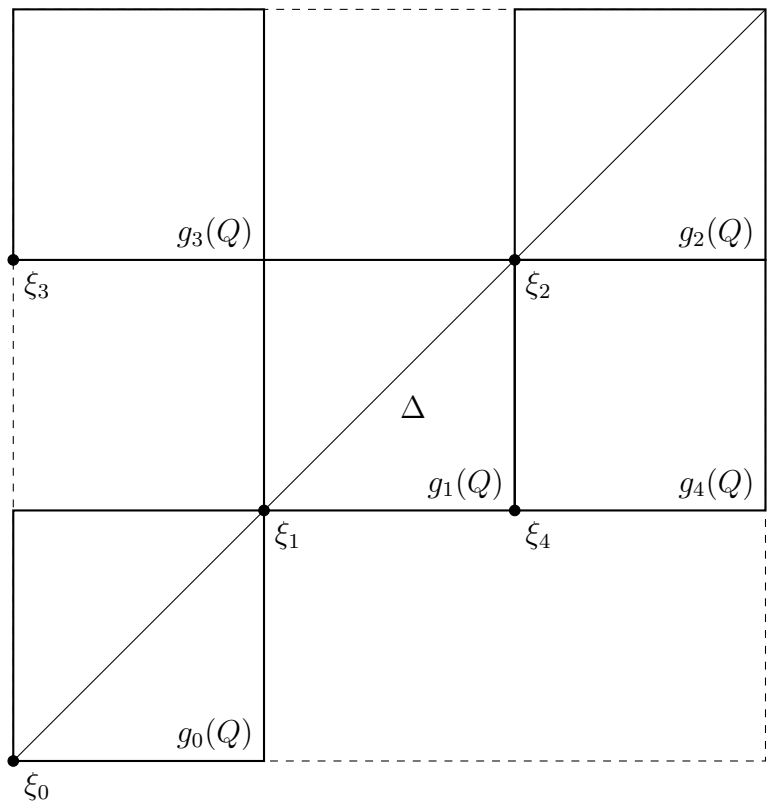
Consider similitudes  $g_i :\rr^2 \to \rr^2$  given by $g_i (x) = x/n + \xi_i$ for $i=0, \ldots, \nu-1$. 
In Figure~\ref{fig:counter2_alt}, those data are depicted for the parameters $n=3$, $\nu = 5$ and 
some special choice of $\xi_3,\xi_4$. 
Let $K$ be the attractor of 
the iterated function system  $(Q,\{g_1, \ldots, g_\nu \})$.
Note that this iterated function system satisfies the open set condition, with the set $V=(0,1)\times (0,1)$, 
i.e. we have  $g_i(V) \subset V$ for every $i = 1,\ldots,\nu$ and $g_i(V)\cap g_j(V)=\emptyset$ for 
every $i\neq j$.
For $\vk =(k_1, \ldots, k_j) \in \Gamma_j$ denote $g_{\vk} = g_{k_1} \circ \cdots \circ g_{k_j}$.
Moreover, let $K_{\vk} = g_{\vk} (K)$, $Q_{\vk} = g_{\vk} (Q)$.
Recall the following facts about the attractor $K$.
\begin{enumerate}
\item  $K = \bigcup_{\ell=1}^\nu K_\ell = \bigcup_{\vk \in \Gamma_j} K_{\vk}$ for each $j\geq 1$.
\item   $K_\vk = \bigcup_{\ell=1}^\nu K_{( \vk,\ell)}$ for $\vk \in \Gamma_j$ and $j\geq 1$.
\item $K = \bigcap_{j=1}^\infty \big( \bigcup_{\vk \in \Gamma_j} Q_{\vk} \big)$.
\item The Hausdorff dimension of $K$ equals 
$d = \log \nu / \log n $, and $0< \mu_d(K) < \infty$, where $\mu_d$ is $d$-dimensional Hausdorff measure.
In addition, 
$\mu_d (K_\vk \cap K_{\vk'} ) = 0$ for $\vk, \vk' \in \Gamma_j $, $\vk \neq \vk'$.
\item Because of the choice of $\xi_1, \ldots, \xi_n$, the diagonal $\Delta$ 
is contained in $K$ and $\Delta = \sum_{\vk \in \Gamma_j^*} g_\vk(\Delta)$; it follows that $\Delta \subset \bigcup_{\vk \in \Gamma_j^*} K_\vk$.
In particular, if ${\vk \in \Gamma_j^*}$ we have  $g_\vk(\Delta) = \Delta \cap Q_\vk = \Delta \cap K_\vk$.
\item \label{it:hd_ifs} For $\vk \in \Gamma_j$ we have ${\rm diam} K_\vk =  1/n^j \cdot {\rm diam} K $ and $\mu_d(K_\vk) =  1/\nu^j \cdot \mu_d(K)$.
\end{enumerate}

%

\begin{figure}
    \includegraphics[width=\textwidth]{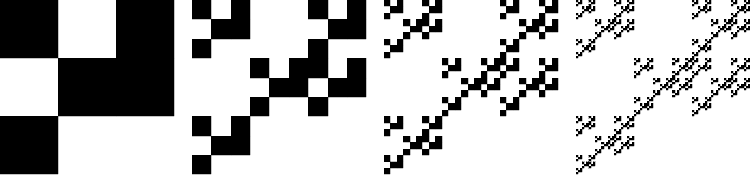}
    \caption{The first four iterations of the fractal $K$, with the same choice 
    of $\nu$ and $(\xi_i)$ as in Figure~\ref{fig:counter2_alt}.}
    \label{fig:ifs2}
\end{figure}

Now, we take $\Omega = K$, $\mathbb P (A)= \mu_d (A\cap K)/ \mu_d(K)$ and 
$\bigcup_{j=1}^\infty \{K_\vk, \vk \in \Gamma_j\}$ as the collection of atoms. 
Each atom $K_\vk$ is split into $\nu$ atoms of the next generation, each having equal measure. 
For some parameters $0<\sigma<p<\infty$,
denote by $V_{p,\sigma}$ the ring variation space corresponding to this splitting.
Moreover, consider a binary refinement of those atoms
 such that at some stage each $K_{(\vk,n+1)} \cup K_{(\vk,n+2)}$ is an atom. Let $\tilde{V}_{p,\sigma}$ 
 be the ring variation space corresponding to this binary filtration. 
 Note that since every atom/ring in the $\nu$-ary filtration is an atom/ring in the 
 corresponding binary filtration, we have the inequality
 \[
    |f|_{V_{\sigma,p}} \leq |f|_{\tilde{V}_{\sigma,p}},\qquad f\in \tilde{V}_{\sigma,p}.
 \]
 We now show that the reverse inequality is in general not true.

For $\vk \in \Gamma_j^*$, put $a_\vk = \charfun_{K_{(\vk, n+1)}} - \charfun_{K_{(\vk, n+2)}}$. 
Moreover, let $\alpha_j = \nu^{j/p} n^{-j/\sigma}$, and 
$$
f_j = \alpha_j \sum_{\vk \in \Gamma_j^*} a_\vk, \qquad \phi = \sum_{j=1}^\infty f_j.
$$
Observe that by item \eqref{it:hd_ifs} of the above list we have $\| a_{\vk} \|_p^p =2 \nu^{-j-1}$ for $\vk \in \Gamma_j^*$.
Let us first check that the function $\phi$ is contained in $L^p$.
Note that the supports of the functions $a_{\vk}$ are disjoint, and therefore 
also the supports of the functions $f_j$.
Therefore, since $\card \Gamma_j^* = n^j$, 
%
\[
\| f_j \|_p^p = 2 \alpha_j ^p n^{j}/\nu^{j+1} = 2 n^{j(1 - p/\sigma)}/\nu.
\]
Consequently,
\[
 \| \phi \|_p^p = \sum_{j=1}^\infty \| f _j \|_p^p \leq 2/\nu \cdot \sum_{j=1}^\infty n^{j(1 - p/\sigma)} < \infty,
\]
since $p/\sigma >1$.

Next, we prove that $| \phi |_{\tilde{V} _{p,\sigma}} = \infty$.
Indeed, by the disjointness of the supports of $a_{\vk}$ again, we can use 
those supports, which are atoms in our binary filtration by construction,
to estimate for fixed $J$,
\begin{eqnarray*}
| \phi |_{\tilde{V} _{p,\sigma}}^\sigma  
\geq \sum_{j=1}^J \sum_{\vk \in \Gamma_j^*} \alpha_j^\sigma \| a_\vk \|_p^\sigma 
= \sum_{j=1}^J n^j \alpha_j^\sigma 2^{\sigma/p} \nu^{-(j+1)\sigma/p} = (2/\nu)^{\sigma/p} \cdot J,
\end{eqnarray*}
implying $ |\phi|_{\tilde{V}_{\sigma,p}} = \infty$.

Lastly, we show that the variation norms of  the functions $f_j$ in the 
space $V_{\sigma,p}$ are uniformly bounded, i.e., $\sup_J | \phi_J |_{{V} _{p,\sigma}} < \infty$.
To this end, observe first that each $f_j$ is 
constant on $K_{\vk}$ if $\vk\in \Gamma_s \setminus \Gamma_s^*$  for some $s$.
Therefore, the same is true for the function 
$\phi_J$. Consequently, if $R \subset K_\vk$ with $\vk \in \Gamma_s \setminus \Gamma_s^*$ for some $s$, then $E_p (f,R) = 0$.

If $\vk \in \Gamma_s^*$ we get, since $p/\sigma > 1$,
\[
\| \phi \|_{L^p(K_\vk)}^p   =  \sum_{j=s}^\infty \sum_{\vz \in \Gamma_j^*: K_\vz \subset K_\vk} \alpha_j^p \| a _\vz \|_p^p
\simeq \sum_{j=s} ^J n^{j-s} n^{-jp/\sigma} \nu^j \nu^{-(j+1)} \simeq n^{-sp/\sigma}.
\]
Consequently, for a ring $R = K_\vk \setminus K_{\vk'}$ with $\vk \in \Gamma_s^*$ and $\vk' \in \Gamma_t$ with $t>s$,
$$
E_p(\phi,R) \leq \| \phi \|_{L^p(K_\vk)} \simeq n^{-s/\sigma}.
$$
Next, note that 
$| \Delta \cap R | \simeq |\Delta \cap K_\vk | = n^{-s} |\Delta| = n^{-s}$, where $| \cdot |$ denotes 
the  1-dimensional Lebesgue measure on the diagonal $\{(t,t): t \in \rr \}$.
Combining this with the  above we find
$$
E_p(\phi,R)^\sigma \lesssim n^{-s} \simeq | \Delta \cap R|.
$$
Furthermore, it is clear that if two rings  $R$ and $R'$ are disjoint, so are the sets 
$\Delta\cap R$ and $\Delta \cap R'$.
Therefore, if $\tilde{\Pi}$ is a collection of pairwise disjoint atoms or rings of the form $R =  K_\vk \setminus B$ with
$\vk \in \Gamma_s^* $ and either $B = \emptyset$ or $B= K_{\vk'}$ with $\vk' \in \Gamma_t$ with $t>s$, we conclude
$$
\sum_{R \in \tilde{\Pi}} E_p(\phi,R) ^\sigma \lesssim \sum_{R \in \tilde{\Pi}} | R \cap \Delta| \leq |\Delta|=1,
$$
yielding that $|\phi|_{V_{p,\sigma}} < \infty$.
Summarizing, this means that the spaces $V_{\sigma,p}$ and $\tilde{V}_{\sigma,p}$ 
do not coincide.

\subsection{\label{ssq.2} Analysis of the variation space $V_{\sigma,p}$}
Because of Theorem \ref{theo.101}, it is of interest to understand the space $V_{\sigma,p}$ better.
For this, for each $f \in L^p(S)$, we are going to give sequences $W_\mu(f)$ such that $\| f - W_\mu f \|_p \to 0$ as $\mu \to \infty$, and simultaneously $|f |_{V_{\sigma,p}} \simeq \sup_{\mu} | W_\mu(f) |_{V_{\sigma,p}}$.

Clearly we have

\begin{prop}\label{prop.Ep.zero}For all $f\in L^p(\Omega)$ and $A\in\mathscr F$, we have
\[
E_p(f,A) = 0 \Leftrightarrow f\charfun_A \in S_A.
\]
\end{prop}

 Let $M \geq 1$.
For each $f \in L^p$ and $A \in \caa$, let $f_A \in S$ be such that 
\begin{equation}\label{txt.eq.3a}
\| f_A - g \|_{L^p(A)} \leq M \| f -g \|_{L^p(A)} \quad \hbox{ for all } g \in S.
\end{equation}
Clearly, $f_A$ depends also on $M$, but we do not show this in the notation.
We next state the following lemma, whose proof is straightforward. 
It provides us with the explicit relation between the constants $M$ in \eqref{txt.eq.3a} and $L$ 
in \eqref{eq:defL}.
\begin{lem}\label{txt.lem.11} Let $0< p < \infty$ and  let $\rho = \min(1,p)$.
Let $f \in L^p$ and $A \in \caa$.

Then, 
\begin{enumerate}[(i)]
\item  If $f_A \in S$ satisfies \eqref{txt.eq.3a}, then 
$$
\| f - f_A \|_{L^p(A)} \leq (M^\rho +1)^{1/\rho} E_p (f,A).
$$

\item  Fix $L \geq 1$, and let $f_A \in S$ be such that 
\begin{equation}\label{eq:defL}
\| f - f_A \|_{L^p(A)} \leq L E_p (f,A).
\end{equation}
Then $f_A$ satisfies \eqref{txt.eq.3a} with $M = (L^\rho +1)^{1/\rho}$.
Consequently, for each $M \geq 2^{1/\rho}$, $f$ and $A$, there exists $f_A \in S$ satisfying \eqref{txt.eq.3a}.
\end{enumerate}
\end{lem}

Now, fix $M \geq 1$. For $f \in L^p$ and $\mu \geq 0$, put
\begin{equation}\label{eq.Wmu.f}
W_{\mu,M} f = W_\mu(f) := \sum_{A\in\caa_\mu} f_A\charfun_A,
\end{equation}
where $f_A$ satisfies  \eqref{txt.eq.3a} with given $M$.

We note that $W_\mu(f)$ is not unique, moreover, $W_\mu$ need not be  a linear operator. And we have guaranteed the existence of  
$W_\mu f$ only for $M$ sufficiently large, cf. Lemma \ref{txt.lem.11} (ii). 

The main result of this section is the following:

\begin{theo}
\label{theo.txt.1}
Let $0 < \sigma <  p < \infty$. Fix $M$ and let $W_\mu$ be given by formula \eqref{eq.Wmu.f} 
with this fixed $M$.

  Then for each $f \in L^p(S)$ we have
$W_\mu f \to f $ in $L^p(\Omega)$.

Moreover, for   any $f\in L^p(S)$
\[
|f|_{V_{\sigma,p}} \leq \liminf_{\mu\to+\infty} |W_\mu f|_{V_{\sigma,p}} \leq \limsup_{\mu \to +\infty} |W_\mu f|_{V_{\sigma,p}} \leq
\sup_{\mu\in\nn} |W_\mu f|_{V_{\sigma,p}} \leq M |f|_{V_{\sigma,p}}.
\]
That is,  if one of the values in the above chain is $\infty$, then all of them are  $\infty$. If one of these values is finite, then all of them are finite.
\end{theo}
We remark that in particular  for  $M=1$, we have
$|f|_{V_{\sigma,p}} =  \lim_{\mu\to+\infty} |W_\mu f|_{V_{\sigma,p}}
= \sup_{\mu\in\nn} |W_\mu f|_{V_{\sigma,p}}$.

We split the proof of Theorem \ref{theo.txt.1} into a series of auxiliary results.

\begin{prop}\label{prop.Wmu.f.est}
 Let $W_\mu$ be given by formula \eqref{eq.Wmu.f} with  fixed $M$. Let $\rho = \min(1,p)$ and $L = (M^\rho+1)^{1/\rho}$.

Then for each $f\in L^p(S)$ and $\mu \to \infty$
\[
\Vert f - W_\mu(f)\Vert_{L^p(\Omega)} \leq L \Big( \sum_{A\in \caa_\mu} E_p(f, A)^p\Big)^{1/p} \to 0.
\]
\end{prop}
\begin{proof} The first inequality is a consequence of Lemma \ref{txt.lem.11} 
  (i) and the definition of $W_\mu f$, cf. \eqref{eq.Wmu.f}.

It is clear that $f \in L^p(S)$ if and only if 
$  \sum_{A\in \caa_\mu} E_p(f,A)^p \to 0$ as $\mu \to \infty$.
\end{proof}

\begin{prop}\label{fact.1}
  Let $0<\sigma<p<\infty$ and $f\in L^p(S)$ with $f_n \to f$ in $L^p(S)$.

Then we have
\begin{equation}
|f|_{V_{\sigma,p}(S)} \leq \liminf_{n\to\infty} |f_n|_{V_{\sigma,p}(S)}. 
\end{equation}
In particular, choosing $f_n = W_n f$,
\[
 |f|_{V_{\sigma,p}(S)} \leq \liminf_{n\to\infty} |W_n f|_{V_{\sigma,p}(S)}. 
\]
\end{prop}
\begin{proof}

We begin with the observation that for each atom or ring $R$ we have 
 $E_p (f_\mu , R) \to E_p (f,R)$ as $\mu \to \infty$,
 which is a consequence of the convergence $f_\mu \to f$ in $L^p$.

Consider first the case  $|f|_{V_{\sigma,p}(S)}<\infty$. Fix  $\varepsilon>0$ and a 
finite family $\crr$ of disjoint atoms or rings such that 
\[
\sum_{R\in \crr} E_p(f,R)^\sigma \geq  |f|_{V_{\sigma,p}(S)}^\sigma - \varepsilon.
\]
Let $r$ be the number of sets in the collection $\crr$ and
take $\mu_0 = \mu_0(r, \varepsilon)$  such that
for all $\mu \geq \mu_0$ and $R\in \crr$,
\[
 E_p(f_\mu,R)^\sigma \geq  E_p( f,R)^\sigma - \varepsilon/r.
\]
This is possible by $E_p(f_\mu,R)\to E_p(f,R)$ as $\mu\to\infty$.
Then for all $\mu\geq \mu_0$ we have
\[
 \sum_{R\in\crr} E_p(f_\mu,R)^\sigma \geq \sum_{R\in\crr} E_p( f,R)^\sigma - \varepsilon \geq |f|_{V_{\sigma,p}(S)}^\sigma - 2\varepsilon
\]
and consequently for all $\mu\geq \mu_0$
\[
 |f_\mu|_{V_{\sigma,p}(S)}^\sigma \geq |f|_{V_{\sigma,p}(S)}^\sigma - 2\varepsilon.
\]
It follows that $\liminf_{\mu \to \infty} |f_\mu|_{V_{\sigma,p}(S)} \geq 
|f|_{V_{\sigma,p}(S)}$.

Next consider the case that  $f\in L^p(S)$ and  $|f|_{V_{\sigma,p}(S)}=\infty$. Fix $m\in\nn$ and select a finite family $\crr$ of pairwise disjoint atoms or rings such that
\[
\sum_{R\in\crr}  E_p( f,R)^\sigma \geq m+1.
\]
Similarly as before, with $r$ being the number 
of sets in the collection $\crr$, we take $\mu_0=\mu_0(r)$ so that 
for all $\mu\geq \mu_0$ and $R\in\crr$,
\[
 E_p(f_\mu,R)^\sigma \geq  E_p( f,R)^\sigma - 1/r.
\]
Then 
\[
 |f_\mu|_{V_{\sigma,p}(S)} ^\sigma\geq  \sum_{R\in\crr}  E_p( f,R)^\sigma - 1 \geq m,
\]
which implies in this case $\liminf_{\mu \to \infty} |f_\mu|_{V_{\sigma,p}(S)} = \infty$.
\end{proof}
\label{eq:Pimu}

In the sequel we will need a family of all atoms and rings generated by 
some finite partition $\mathscr U$ of $\Omega$ into disjoint rings or 
atoms corresponding to the filtration $(\mathscr F_n)$. Given such $\mathscr U$,
we first define an associated $\sigma$-algebra $\mathscr F_{\mathscr U}$ as follows.
We write each $U\in\mathscr U$ as $U = A_U \setminus B_U$ with the possibility 
of $B_U = \emptyset$ when $U$ is an atom. If $B_U = \emptyset$, let 
$\mu = \mu(U)$ be minimal such that $A_U\in\mathscr A_\mu$.
If $B_U \neq \emptyset$, let $\mu=\mu(U)$ be minimal such that $B_U\in\mathscr A_\mu$.
Then, define the $\sigma$-algebra $\mathscr F_{\mathscr U}$ to be such that 
\[
  \mathscr F_{\mathscr U} \cap U = \mathscr F_{\mu(U)} \cap U\qquad \text{for each $U\in\mathscr U$.}  
\]

Next, we define the collection $\Pi_{\mathscr U}$ of all rings or atoms coarser than $\mathscr U$
 by saying that the set $X\subset \Omega$
is contained in $\Pi_{\mathscr U}$ if and only if 
\begin{enumerate}
  \item either $X\in\mathscr A$ and $X$ is a strict superset of an atom of $\mathscr F_{\mathscr U}$
  \item or $X$ is a ring $X=A\setminus B$ with $A$ being a strict superset of an atom 
      of $\mathscr F_{\mathscr U}$ and $B$ being a superset of an atom of $\mathscr F_{\mathscr U}$.
\end{enumerate}

\begin{lem}\label{lem.fact.2.gen}
  Let $\mathscr U$ be a finite partition of $\Omega$ into 
  disjoint rings or atoms.
  Define 
\[
f = \sum_{U\in\mathscr U} f_U, \hskip1cm \text{for some functions }f_U\in S_{U}.
\]
Then
\[
|f|_{V_{\sigma,p}(S)} = \max_{\crr\in \crr(\mathscr U)} \left( \sum_{R\in\crr} E_p(f,R)^\sigma \right)^{1/\sigma},
\]
where $\crr({\mathscr U})$ is the set of families of disjoint elements of $\Pi_{\mathscr U}$.
\end{lem}
\begin{proof}

  Let $\cpp$ be any finite family of atoms or rings that are pairwise disjoint. Instead of $\cpp$, it is enough to consider
\[
\cpp' = \{ R\in \cpp : E_p(f,R)\neq 0\}.
\]
We distinguish the cases of atoms and rings contained in $\mathcal P'$.
\begin{enumerate}
\item Assume that  $Q\in \mathcal P'$ is an atom such that 
$Q\notin \Pi_{\mathscr U}$, which means that $Q$ is a 
 subset of some atom of $\mathscr F_{\mathscr U}$.
Then we have by definition 
of $\mathscr F_{\mathscr U}$ that $Q\subseteq U$ for some $U\in\mathscr U$.
This implies that $E_p(f,Q)=0$, but this is not possible by definition of $\mathcal P'$.
Therefore, we have that $Q\in \Pi_{\mathscr U}$.

\item Let  $R$ be a ring such that $R\in \cpp'$ with   $R=A_R\setminus B_R$.  Then  $A_R\in\Pi_{\mathscr U}$
by the same reason as in (1). 
Next we distinguish the possibilities for $B_R$.
\begin{enumerate}
\item If $B_R$ is a superset of some atom of $\mathscr F_{\mathscr U}$ as well, 
then the ring $R\in\Pi_{\mathscr U}$. 
\item If $B_R$ is a strict subset of some atom of $\mathscr F_{\mathscr U}$, 
then there is no atom or ring $Q\in \cpp'$ such that $Q\subseteq B_R$. It follows that
$A_R$ is disjoint from all elements of $\cpp'$ other than $R$.
Therefore, the modified family $\cpp'_R = (\cpp' \setminus \{R\}) \cup \{A_R\}$
is again a family of pairwise disjoint rings or atoms.
As $ E_p(f,R)^\sigma \leq E_p(f,A_R)^\sigma$, we have 
$$
\sum_{U \in \cpp'} E_p(f,U)^\sigma \leq  \sum_{U \in \cpp'_R} E_p(f,U)^\sigma .
$$
Applying the above procedure to each $R \in \cpp'$ with $R \not
\in \Pi_{\mathscr U}$, we arrive at a family $\tilde \cpp'\subset\Pi_{\mathscr U}$ -- consisting of pairwise disjoint atoms and rings -- and such that
\[
\sum_{R\in \cpp} E_p(f,R)^\sigma = \sum_{R\in \cpp'} E_p(f,R)^\sigma \leq \sum_{R\in \tilde \cpp'} E_p(f,R)^\sigma. \qedhere
\]
\end{enumerate}
\end{enumerate}
\end{proof}

\begin{prop}\label{fact.2}
  Let $0<\sigma < p<\infty$.
For all $f\in V_{\sigma,p}(S)$ and $\mu \geq 0$, we have 
\begin{equation}
 |W_\mu f|_{V_{\sigma,p}(S)} \leq  M \cdot |f|_{V_{\sigma,p}(S)} .
\end{equation}
Moreover,  if $\sup_{\mu \geq 0} |W_\mu f|_{V_{\sigma,p}(S)} = \infty$, then $ |f|_{V_{\sigma,p}(S)} = \infty$.
\end{prop}
\begin{proof}
Fix $\mu \geq 0$ and $R \in \Pi_{\mathscr A_\mu}$. 
Note that $R =  \bigcup_{A\in \caa_\mu, A\cap R\neq\emptyset} A$.
Applying the definition of  $W_\mu f$, we get for all $g\in S$
\begin{eqnarray*}
E_p(W_\mu f,R) & \leq & \| W_\mu f - g \|_{L^p(R)} 
= \Big( \sum_{A\in \caa_\mu, A\cap R\neq\emptyset} \| (f_A - g)\charfun_A \|_p^p \Big)^{1/p}
\\ & 
\leq & M \Big( \sum_{A\in \caa_\mu, A\cap R\neq\emptyset} \| (f - g)\charfun_A \|_p^p \Big)^{1/p} 
= M \| f -g \|_{L^p(R)} .
\end{eqnarray*}
This implies
$$
E_p(W_\mu f,R) \leq M E_p( f,R).
$$
Therefore, applying Lemma \ref{lem.fact.2.gen} in the special 
case $\mathscr U = \mathscr A_\mu$, we find
$$
|W_\mu f|_{V_{\sigma,p}(S)} \leq M \max_{\crr\in \crr(\mathscr A_\mu)} \left( \sum_{R\in\crr} E_p(f,R)^\sigma \right)^{1/\sigma}.
$$
This is enough to finish the proof.
\end{proof}

We are now ready to prove the main result of this section:
\begin{proof}[Proof of Theorem \ref{theo.txt.1}]
Theorem \ref{theo.txt.1} is a consequence of Propositions \ref{prop.Wmu.f.est}, \ref{fact.1}, and \ref{fact.2}.
\end{proof}

We conclude Section~\ref{ssq.2} by  giving some more explicit examples of $W_\mu$ 
in case of $1 \leq p < \infty$.
 In particular, it is possible to then set $W_\mu = P_\mu$,
 where $P_\mu$ denotes the linear orthoprojector onto the space $S_\mu$ of 
 functions contained piecewise in $S$ on each atom of $\mathscr F_\mu$.
 The properties of $P_\mu$ were already analyzed in \cite{part1,part2}. In particular,
 Theorem~4.4 of \cite{part1} states that
 there exists a constant $M$, depending only on the  constants $c_1,c_2$ from 
 the stability condition \eqref{eq:L1Linfty}, such that 
 \begin{equation}\label{eq:Pmu}
    \| P_\mu : L^p \to (S_\mu,\|\cdot\|_p)\| \leq M.
 \end{equation}
 This implies that $f_A := P_\mu(f)|_{A}$ satisfies  \eqref{txt.eq.3a}
 with the same constant $M$, and 
Proposition~\ref{fact.2}
implies
$$\| P_\mu : V_{\sigma,p} \to  V_{\sigma,p} \| \leq M,$$ with
$\| f \|_{ V_{\sigma,p} } = \| f \|_p + | f |_{ V_{\sigma,p} } $.
Moreover, since the orthoprojectors $P_\mu$ satisfy 
$P_\mu P_\nu  = P_\nu P_\mu = P_{\min(\mu,\nu)}$, we also obtain 
\begin{equation}\label{eq:nested}
|P_\mu f|_{V_{\sigma,p}} \leq M |P_\nu f|_{V_{\sigma,p}}, \qquad \mu\leq \nu.
\end{equation}
By specialization of Theorem~\ref{theo.txt.1} to this linear case, we 
get

\begin{co}\label{co:liminf}
Let $1\leq p<\infty$ and  $0 < \sigma <  p$. 

Then, for   any $f\in L^p(S)$
\[
|f|_{V_{\sigma,p}} \leq \liminf_{\mu\to+\infty} |P_\mu f|_{V_{\sigma,p}} \leq \limsup_{\mu \to +\infty} |P_\mu f|_{V_{\sigma,p}} \leq
\sup_{\mu\in\nn} |P_\mu f|_{V_{\sigma,p}} \leq M |f|_{V_{\sigma,p}},
\]
with the same constant $M$ as in \eqref{eq:Pmu}.
\end{co}

In particular, in the case of piecewise constant functions, i.e.
with $S = \operatorname{span}\{\charfun_{\Omega}\}$,
the operator $P_\mu$ is the conditional expectation $\mathbb E_\mu$ with respect to $\mathscr F_\mu$,
satisfying \eqref{eq:Pmu} with the constant $M = 1$, giving equality of all the terms 
in Corollary~\ref{co:liminf}. Additionally, by \eqref{eq:nested}, the sequence $(|\mathbb E_\mu f|_{V_{\sigma,p}})_\mu$
is non-decreasing.

\subsection{Variation spaces need not be separable}
In this section we show that our variation spaces $V_{\sigma,p}(S)$,
similarly as the space $\operatorname{BV}(\rr^d)$,
are not separable in general. This implies in particular that we cannot expect a general convergence 
result of the form $| W_\mu f - f|_{V_{\sigma,p}} \to 0$ as $\mu\to\infty$, 
cf. Theorem~\ref{theo.txt.1}.

We start by proving the non-separability of $V_{\sigma,p}$ in the dyadic Haar case.
That means, we consider $\Omega = [0,1]$ with the Borel $\sigma$-algebra and Lebesgue measure $\lambda = |\cdot|$. Moreover,
let $\mathscr D_n$ be the dyadic $\sigma$-algebra of order $n$ and
$\mathscr A_n$ be the collection of all atoms $A$ of $\mathscr D_n$.
Denote by $\mathbb E_n$ the conditional expectation with respect to the 
$\sigma$-algebra $\mathscr D_n$.
We consider the space $S = \operatorname{span}(\charfun_{[0,1]})$.
Note that in this section, we switch to a level-wide refinement of atoms 
in the passage from $\mathscr A_{n-1}$  to $\mathscr A_n$, instead of 
our usual refinement of only one atom.

For atoms $A\in\mathscr A_n$, we have $|A| = 2^{-n}$.
Let $(r_n)_{n=1}^\infty$ be the collection of Rademacher functions so
that $r_n$ is $\mathscr D_n$-measurable, but not $\mathscr D_{n-1}$-measurable.
Let $u:[0,1]\to \mathbb R$ be a measurable function. Then, by a change of
variables, we have for each atom $A\in \mathscr A_n$
\begin{equation}
    \int_A u\Big( \sum_{j=n+1}^\infty a_{j} r_j \Big) d\lambda = 2^{-n} \int_0^1 u\Big( \sum_{j=1}^\infty a_{j+n} r_j\Big) d\lambda.
\end{equation}
In particular, if we take $u(x) = |x|^p$ for $0<p<\infty$, we have by Khintchine's
inequality
\begin{equation}\label{eq:local_khintchine}
    \Big( \int_A \big|\sum_{j=n+1}^\infty a_j r_j \big|^p d\lambda \Big)^{1/p} \simeq 2^{-n/p} \Big( \sum_{j=n+1}^\infty |a_{j}|^2\Big)^{1/2}.
\end{equation}
In the present setting, $ E_p(f,R) = \inf_{g\in S} \| f - g\|_{L^p(R)} = \inf_{c\in \mathbb R} \big(\int_R | f - c|^p\big)^{1/p}$.
Then, we note that for every $0<p<\infty$ and every $f = \sum_{j=1}^\infty a_j r_j$, 
\begin{equation}\label{eq:bestappr0}
    \|f\|_p \lesssim E_p(f,[0,1]).
\end{equation}
Indeed, since $f$ and $-f$ have the same distribution, we obtain for every $c\in\mathbb R$
\[
    \|f\|_p = \big\| \frac{f}{2} + \frac{c}{2}  +\frac{f}{2} - \frac{c}{2} \big\|_p \lesssim
    \frac{1}{2} \big( \| f+c\|_p + \|f-c\|_p \big) = \|f-c\|_p.
\]
By scaling, we get from \eqref{eq:bestappr0} for every $n\in\mathbb N$ and every atom $A\in \mathscr A_{n-1}$ that
\begin{equation}\label{eq:bestappr}
    \Big\| \sum_{j=n}^\infty a_jr_j \Big\|_{L^p(A)}\lesssim
    E_p\Big( \sum_{j=n}^\infty a_j r_j, A\Big) =
    E_p\Big( \sum_{j=1}^\infty a_j r_j, A\Big).
\end{equation}

\begin{prop}\label{prop:counter2}
    Let $0<\sigma <p<\infty$ and $\beta = 1/\sigma - 1/p > 0$. Let 
    $\Lambda \subset \mathbb N$ arbitrary.
     Define the function
    \begin{equation}\label{eq:fct_f}
        f_\Lambda = \sum_{j\in\Lambda} 2^{-j\beta} r_j.
    \end{equation}

    Then we have
    \begin{enumerate}
        \item $f_\Lambda \in V_{\sigma,p}$,
        \item for every $\Gamma\subset \mathbb N$ with 
         $\Gamma \neq \Lambda$, 
         we have $| f_\Gamma - f_\Lambda|_{V_{\sigma,p}} \gtrsim 1$,
          with a constant independent of $\Gamma$ or $\Lambda$. 
    \end{enumerate}
\end{prop}
\begin{proof}
    We begin with the proof of (1). Let $\Pi$ be an arbitrary partition of
    $[0,1]$ consisting of disjoint rings and atoms. For $R\in \Pi$, we let
    $A,B\in \mathscr A:=\cup_n \mathscr A_n$ such that $R = A\setminus B$ if $R$
    is a ring and $R = A\in\mathscr A$ if $R$ is an atom. For this set $R$,
    let $n=n(A)$ be the smallest index such that $A\in\mathscr A_n$.

    Then we
    estimate
    \begin{align*}
        E_p(f_\Lambda,R) & \leq \|f_\Lambda - \mathbb E_n f_\Lambda\|_{L^p(R)} \leq \|f_\Lambda-\mathbb E_nf_\Lambda\|_{L^p(A)}         \\
                 & = \Big\| \sum_{j\in\Lambda, j>n}2^{-j\beta} r_j\Big\|_{L^p(A)}  \simeq 2^{-n/p}
        \Big(\sum_{j\in\Lambda, j>n}2^{-2j\beta} \Big)^{1/2},
    \end{align*}
    by Khintchine's inequality \eqref{eq:local_khintchine}. Since the latter sum is
    a geometric series, we further obtain
    \begin{align*}
        E_p(f_\Lambda,R) \lesssim 2^{-n/p} 2^{-n\beta} =2^{-n/\sigma} = |A|^{1/\sigma}\leq (2|R|)^{1/\sigma},
    \end{align*}
    where the latter inequality is clear for dyadic partitions. Summing this inequality over $R\in\Pi$
    yields
    \[
        \sum_{R\in\Pi} E(f,R)^\sigma \lesssim \sum_{R\in\Pi} |R| =1.
    \]
    Since this is true for every partition $\Pi$ of $[0,1]$ into rings or atoms, we get
    \[
        |f|_{V_{\sigma,p}} \lesssim 1,
    \]
    concluding the proof of (1).

    Next, we prove (2) and let $\Gamma\subset \mathbb N$ with $\Gamma\neq \Lambda$ arbitrary.
    Let $n$ be the smallest integer that is present in exactly one of the sets $\Gamma$ or $\Lambda$.
    Phrased differently, $n$ is the smallest number contained in the symmetric 
    difference $\Gamma\Delta \Lambda$ of $\Gamma$ and $\Lambda$.
To estimate $|f_\Gamma - f_\Lambda|_{V_{\sigma,p}}$ from
    below, we choose the particular partition $\mathscr A_{n-1}$ of atoms in $\mathscr D_{n-1}$
    and write 
    \begin{align*}
        |f_\Gamma-f_\Lambda|_{V_{\sigma,p}}^\sigma & \geq  \sum_{A\in\mathscr A_{n-1}} E_p(f_\Gamma - f_\Lambda, A)^\sigma        
                                                    \gtrsim \sum_{A\in\mathscr A_{n-1}} \|f_\Gamma - f_\Lambda\|_{L^p(A)}^\sigma,
    \end{align*}
    where the latter inequality follows from \eqref{eq:bestappr}. We continue and
    estimate
    \begin{align*}
        |f_\Gamma - f_\Lambda|_{V_{\sigma,p}}^\sigma & \gtrsim \sum_{A\in\mathscr A_{n-1}}
        \Big\| \sum_{j\in \Gamma\Delta\Lambda} 2^{-j\beta} r_j \Big\|_{L^p(A)}^\sigma \simeq
        \sum_{A\in\mathscr A_{n-1}}  2^{-n\sigma/p} \Big(\sum_{j\in\Gamma\Delta\Lambda} 2^{-2j\beta}\Big)^{\sigma/2}
    \end{align*}
    by \eqref{eq:local_khintchine}. Since $n\in\Gamma\Delta \Lambda$, we 
    further obtain
    \begin{align*}
        |f_\Gamma-f_\Lambda|_{V_{\sigma,p}}^\sigma \gtrsim \sum_{A\in\mathscr A_{n-1}} 2^{-n\sigma/p} 2^{-n\sigma\beta}
        = 2^{n-1} 2^{-n\sigma/p} 2^{-n\sigma\beta} = 1/2,
    \end{align*}
    concluding the proof of (2).
\end{proof}

As a corollary, we get 
\begin{theo}\label{thm:sep}
  For every choice of parameters $0 < \sigma < p < \infty$,
  the space $V_{\sigma,p}(S)$ for $S = \operatorname{span}\{\charfun_{[0,1]}\}$ and 
  the dyadic filtration $(\mathscr D_n)$, is not separable.
\end{theo}
\begin{proof}
  By the above proposition, the uncountable 
  set of functions $\{ f_\Lambda : \Lambda\subset \mathbb N\}$ 
  is contained in $V_{\sigma,p}$ and two different functions 
  from this set have a uniformly positive distance to 
  each other in $V_{\sigma,p}$.
\end{proof}

We now come back to the case of general binary filtrations $(\mathscr F_n)$ 
and finite-dimensional spaces $S\subset L^\infty(\Omega)$ satisfying the 
stability condition~\eqref{eq:L1Linfty}. It is possible 
to generalize the result of Propositon~\ref{prop:counter2} and thus 
Theorem~\ref{thm:sep}, as follows:
\begin{theo}\label{thm:not_sep}
  Let $(\mathscr F_n)$ be a binary filtration with 
  $\lim_{n\to\infty}\max_{A\in\mathscr A_n} |A| = 0$
  and let $S\subset L^\infty(\Omega)$
  be finite-dimensional satisfying~\eqref{eq:L1Linfty}.
  Assume additionally that  there exists a constant $L$ such that each chain $(X_i)_{i=1}^L$
  of atoms of length $L$ satisfies $|X_L| \leq |X_1| / 2$.

  Then, for every choice of parameters $0 < \sigma < p < \infty$,
   the corresponding variation space $V_{\sigma,p}(S)$ is not separable.
\end{theo}

The proof uses the same ideas as the proof of Theorem~\ref{thm:sep},
but is more technical. The details can be found in Appendix~\ref{app:sep}.

\section{Embeddings between Approximation and variation spaces \label{sec:embedding}}
In this section we give a partial answer to the question whether 
the K-functional between the spaces $L^p$ and $V_{\sigma,p}$ is 
attained for greedy approximation and continue to use $\nu$-ary filtrations $(\mathscr F_n)$.
Together with a $\nu$-ary filtration $(\mathscr F_n)$, 
we consider a local orthonormal system $\Phi$ such that for each 
$1 < p < \infty$, its $p$-renormalizations is a greedy basis in
 $L^p(S)$. Recall that then the $p$-renormalized 
 version of $\Phi$ satisfies $p$-Temlyakov property in $L^p(S)$, cf. Section \ref{sec:loc}.

Recall that if $\mathcal Y$ is continuously embedded in $\spp$, 
the K-functional between $\mathcal X$ and $\mathcal Y$ is given by
$K(f,t) = \inf\{ \|f - g\|_{\spp} + t\|g\|_{\mathcal Y} : g\in \mathcal Y\}$
and 
for the norm $\|f\|_{\theta,q}$ in the interpolation space 
$(\spp,\mathcal Y)_{\theta,q}$ we have, for every fixed $r>0$,
the equivalence
\begin{equation}\label{eq:norm_int_space}
\| f \|_{\theta,q} \simeq  \begin{cases}
	\Big(\sum_{n=0}^\infty \big( 2^{nr\theta} K(f,2^{-nr}) \big)^{q}\Big)^{1/q},& \text{if $0<q<\infty$}, \\
	\sup_{n\geq 0} 2^{nr\theta} K(f,2^{-nr}),& \text{if $q=\infty$}.
\end{cases}	
\end{equation}
For this equivalence, see for instance \cite[Chapter 6, equation (7.6)]{constr.approx}.

Following arguments contained in \cite{cdpx.1999} (cf. the proof of Theorem 9.2
 in \cite{cdpx.1999}, or its $d$-variate counterpart, i.e. Theorem 12 of 
\cite{pw.2003}), it is possible to extract the following 
abstract result about when the K-functional is attained for greedy approximation.
\begin{lem}\label{lem:abstract}
    Assume that $(\spp,\|\cdot \|)$ is a Banach space, $\Psi = (\psi_n)$ a greedy basis in $\spp$, and
    let $\mathcal Y\subset \spp$ a continuously embedded semi-quasi-normed space such that $\psi_n\in \mathcal Y$
    for every $n$. Let $\beta >0$.

    Then, the uniform equivalence 
    \[
        K(f,n^{-\beta},\mathcal X,\mathcal Y) \simeq \|f-\mathscr G_nf\| + n^{-\beta} |\mathscr G_n f|_{\mathcal Y},       \qquad n\in\mathbb N,
    \] 
    is satisfied if and only if the following three conditions are true: 
    \begin{enumerate}
        \item There exists a constant $C_1$ such that
              \[
                  \|f - \mathscr G_n f\| \leq \frac{C_1}{n^\beta} |f|_{\mathcal Y},\qquad f\in \mathcal Y.
              \]
        \item There exists a constant $C_2$ such that
              \[
                  |g|_{\mathcal Y} \leq C_2 n^\beta \|g\|,\qquad g\in \Sigma_n^{\Psi}.
              \]
        \item There exists a constant $C_3$ such that
              \[
                  |\mathscr G_n g|_{\mathcal Y} \leq C_3 |g|_{\mathcal Y}, \qquad g\in \mathcal Y.
              \]
    \end{enumerate}
\end{lem}

This lemma, in conjunction with Fact~\ref{intro.f3}, easily implies 
the following result for approximation spaces.
\begin{co}\label{thm:KA}
    Let $\spp$ be a Banach space with unconditional basis $\Psi=(\psi_n)$ satisfying the $p$-Temlyakov property.
    Assume that  $\beta,q \in (0,\infty)$. Then, the K-functional $K(f,t)=K(f,t,\cxx,\mathcal A_q^\beta(\cxx,\Psi))$
    is attained for greedy approximation:
    \begin{equation}\label{eq:KA}
        K(f,n^{-\beta})\simeq \|f-\mathscr G_nf\|_\cxx + n^{-\beta} \|\mathscr G_n f\|_{\mathcal A_q^\beta}.
    \end{equation}
\end{co}

Now we fix a local orthonormal system $\Phi$ described in Section~\ref{sec:loc} corresponding 
to the $\nu$-ary  filtration, which -- after a suitable renormalization -- is a greedy basis
 in $L^p(S)$, $1 < p < \infty$, and satisfies $p$-Temlyakov property in $L^p(S)$.  
In the rest of this section we always consider approximation spaces $\app_q^\alpha 
= \app_q^\alpha(L^p(S), \Phi)$ with respect 
to the space $L^p(S)$ and the dictionary $\Phi$, which is a greedy basis in $L^p(S)$ 
satisfying the $p$-Temlyakov property.

\begin{rem}\label{remark.weak}
    Recall the notation $S_j$ for the space of functions that are, locally on atoms 
    of $\mathscr F_j$, given by functions contained in $S$. Moreoever, 
    $P_j$ denotes the orthoprojector onto $S_j$.
    Corresponding to the local orthonormal basis $\Phi$, it is natural to 
    consider weak greedy algorithms corresponding to 
    whole blocks $(P_j - P_{j-1}) f$ instead of projections onto 
    single functions $\langle f,\varphi_n\rangle \varphi_n$ for $\varphi_n\in\Phi$.

    To wit, let $\Gamma_j$ be the set of indices $n$ such that $\varphi_n$ is contained in the 
    range of $W_j = P_j - P_{j-1}$. Consider $\mu_j(f) = \max_{n\in\Gamma_j}|\langle f,\varphi_n\rangle|$
    and choose a parameter $0<t\leq 1$. Consider a weak greedy algorithm given by a 
    sequence $(j_k)$ such that
    \[
        \min_{1\leq \ell\leq k}\mu_{j_\ell}(f) \geq t \max_{\ell\geq k+1} \mu_{j_\ell}(f)
    \]
    and set 
    \[
        B_k f = \sum_{\ell=1}^k W_{j_\ell} f.    
    \]

   The results of Lemma~\ref{lem:abstract} and Corollary~\ref{thm:KA} then remain 
   valid if $\mathscr G_n f$ is replaced by $B_n f$. For a proof of this fact, we 
   refer to Appendix~\ref{sec:appendix_quasi_greedy}.
\end{rem}

We don't know whether the result of Corollary~\ref{thm:KA} 
extends also to the variation spaces $V_{\sigma,p}$ instead of approximation spaces, i.e. we don't know 
if we have  the uniform equivalence
\begin{equation}\label{eq:conj_K_var}
K(f,2^{-n\alpha}, L^p(S), V_{p,\sigma}) \simeq \|f - \mathscr G_{2^n}f\|_p + 2^{-n\alpha} \|\mathscr G_{2^n} f\|_{V_{\sigma,p}}
\end{equation}
for all $n$.
For this, it would suffice that greedy approximation is bounded in variation space, 
i.e. $|\mathscr G_n f|_{V_{\sigma,p}} \lesssim \|f\|_{V_{\sigma,p}}$,
cf. Lemma~\ref{lem:abstract} with the choices $\mathcal X = L^p(S)$, 
$\mathcal Y = V_{\sigma,p}$, and $\Psi = \Phi$. Indeed, we will see in 
\eqref{eq:jacksonVsigma}  
that (1) of Lemma~\ref{lem:abstract} is true and Proposition~\ref{th.kfun.c} (ii) implies 
(2) of Lemma~\ref{lem:abstract}. Therefore, in order to show \eqref{eq:conj_K_var},
it suffices to show (3) of Lemma~\ref{lem:abstract}.

On the other hand, we next prove some embeddings between the variation spaces 
$V_{\sigma,p}$ and the approximation spaces $\app_q^\alpha$ (cf. Sections~\ref{sec:embed1} and \ref{sec:embed2})
that allows us to conclude a weak version of Corollary~\ref{thm:KA} for variation
spaces. This result is given in Theorem~\ref{thm:K_equiv}.
In the embedding result of Section~\ref{sec:embed2} (and therefore also in Theorem~\ref{thm:K_equiv}), it is crucial to 
assume condition $\operatorname{w2^*}$ from \eqref{eq:introw2s} on 
the geometry of the atoms $\mathscr A$ and the space $S$.
Finally, Proposition~\ref{prop:notw3} provides an example 
showing that boundedness of greedy approximation in variation space is not
true unless we assume a condition similar to $\operatorname{w2}^*$.
This suggests that such a geometric condition is somewhat natural 
in the above embedding results.

\subsection{The embedding of $\mathcal A_q^\alpha$ into $V_{\sigma,p}$}
\label{sec:embed1}

We have the following result:
\begin{theo}\label{thm:embed_app_var}
    Let $1<p<\infty$, $0 < \sigma < p$ and set $\alpha = 1/\sigma - 1/p$.
    Letting  $q = \min(\sigma,2)$, 
    the approximation space $\mathcal A_q^\alpha$ 
    embeds continuously into the variation
    space  $V_{\sigma,p}$.
\end{theo}

In the proof  we need the concept of (Rademacher) type $\rho$ (see for instance \cite{woj.1991}). We recall that 
a Banach space $\cxx$ is of type $\rho, 1\leq \rho\leq 2$, if there exists a constant
$T$ such that
\begin{equation}\label{eq:def_type}
    \operatorname{Ave}_\varepsilon \Big\| \sum_i \varepsilon_i x_i\Big\| \leq T \Big(\sum_i \|x_i\|^\rho\Big)^{1/\rho},\qquad x_i\in \cxx,
\end{equation}
where the average is taken over all choices of signs $\varepsilon_i \in \{\pm 1\}$.
It is known that $L^p$ is of type $\min(p,2)$ and clearly, inequality \eqref{eq:def_type}
for $\rho=r_1$ implies inequality \eqref{eq:def_type} for all $\rho=r_2 < r_1$. 

\begin{proof}
Let $\Phi = (\varphi_j)$ be $L^p$-normalized.
In the proof, we distinguish
between the cases $\sigma\leq 2$ and $\sigma >2$.

\textsc{Case I: $\sigma \leq 2$.}

Let $\Pi$ be a collection of disjoint atoms or rings and let
$ g = \sum_{n\in\Lambda} c_n \varphi_n$ with a finite index set $\Lambda$.
For fixed $R\in\Pi$, let $\Lambda_R = \{ n\in\Lambda : E_p(\varphi_n,R)\neq 0\}$ and
set $g_R = \sum_{n\in\Lambda_R} c_n\varphi_n$.
Then,
\begin{align*}
    \sum_{R\in\Pi} E_p(g,R)^\sigma = \sum_{R\in\Pi} E_p(g_R,R)^\sigma \leq \sum_{R\in\Pi} \|g_R\|_p^\sigma 
    = \sum_{R\in\Pi} \Big\| \sum_{n\in\Lambda_R} c_n\varphi_n \Big\|_p^\sigma.
\end{align*}
Since $(\varphi_n)$ is unconditional in $L^p$, we further obtain
\begin{equation}\label{eq:Ep}
    \sum_{R\in\Pi} E_p(g,R)^\sigma \lesssim \sum_{R\in\Pi} \Big(\operatorname{Ave}_\varepsilon \Big\| \sum_{n\in\Lambda_R} \varepsilon_nc_n\varphi_n \Big\|_p\Big)^\sigma.
\end{equation}
Since $\sigma < p$ and $\sigma\leq 2$, we have that $\sigma \leq \min(p,2)$, which is the type of  $L^p$.
Therefore, we can use inequality \eqref{eq:def_type} for $\rho = \sigma$ to get 
\[
    \sum_{R\in\Pi} E_p(g,R)^\sigma   \lesssim \sum_{R\in\Pi} \sum_{n\in\Lambda_R} |c_n|^\sigma
    = \sum_{n\in\Lambda} |c_n|^\sigma \cdot |\{ R\in\Pi : n\in\Lambda_R\}| \leq \nu \sum_{n\in\Lambda}|c_n|^\sigma,
\]
where in the last step, we used Lemma~\ref{lem.bernst}, combined with the observation that $\Phi \subset \Sigma_\nu^\cc$, to estimate the cardinality
of the set in question  by $\nu$.
Therefore, using also Fact~\ref{intro.f3}, we have the embedding
of  $\mathcal A_\sigma^\alpha$
into $V_{\sigma,p}$ for $\alpha = 1/\sigma - 1/p$ and parameters  $\sigma\leq 2$.

\textsc{Case II: $2<\sigma<p$.}
Let $f\in \mathcal A_2^\alpha$ (for $\alpha = 1/\sigma - 1/p$) with
the expansion $ f = \sum_n c_n \varphi_n$. 
Let $(a_\ell\psi_\ell)$ be the decreasing rearrangement of 
$(c_n \varphi_n)$ with respect to the $p$-norm $\| c_n \varphi_n\|_p = |c_n|$.
Define
\[
    h_j = \sum_{\ell = 2^j +1}^{2^{j+1}}
    a_\ell \psi_\ell.
\]
Let $\Pi$ be a collection of disjoint atoms or rings. 
Write
\begin{align*}
    \Pi_j & = \{ R\in \Pi : \text{ there exists } 
    \ell \text{ with } 2^{j} < \ell \leq 2^{j+1} \text{ such that } E_p(\psi_\ell,R) \neq 0\}.
\end{align*}
By Lemma~\ref{lem.bernst} and  $\Phi \subset \Sigma_\nu^\cc$, we get $\card\Pi_j \leq  \nu 2^j$.
For fixed $R\in \Pi$, set
\[
    f_R := \sum_{\ell : E_p(\psi_\ell,R)\neq 0}  a_\ell \psi_\ell = \sum_j \sum_{2^j < \ell \leq 2^{j+1}, E_p(\psi_\ell,R)\neq 0}  a_\ell\psi_\ell  =: \sum_j h_{j,R}.
\]
Next,
\[
    E_p(f,R) = E_p(f_R, R) \leq \Big\| \sum_{j} h_{j,R} \Big\|_p  \lesssim \Big(\sum_j \|h_{j,R}\|_p^2\Big)^{1/2},
\]
since $\Phi$ is unconditional in $L^p$ and $L^p$ is of type $2$. Denoting by $n_{j,R}$ the
number of indices $\ell$ with $2^j < \ell \leq 2^{j+1}$ and  $E_p(\psi_\ell,R)\neq 0$, we further estimate
\[
    \Big(\sum_j \|h_{j,R}\|_p^2\Big)^{1/2} \lesssim \Big( \sum_j a_{2^j}^2 n_{j,R}^{2/p} \Big)^{1/2},
\]
where we used unconditionality and $p$-Temlyakov property \eqref{eq:temlyakov} of $\Phi$. Therefore,
\begin{align*}
    \Big( \sum_{R\in \Pi} E_p(f,R)^\sigma \Big)^{1/\sigma} \lesssim
    \Big[ \sum_{R\in\Pi}\Big( \sum_j a_{2^j}^2 n_{j,R}^{2/p} \Big)^{\sigma/2} \Big]^{1/\sigma}.
\end{align*}
Since $\sigma\geq 2$, we use the triangle inequality in $\ell_{\sigma/2}$ to deduce
\begin{equation}\label{eq:E}
    \Big( \sum_{R\in \Pi} E_p(f,R)^\sigma \Big)^{1/\sigma} \lesssim
    \Big[ \sum_j \Big( \sum_{R\in \Pi_j} \big( a_{2^j}^2 n_{j,R}^{2/p} \big)^{\sigma/2} \Big)^{2/\sigma} \Big]^{1/2}
    = \Big[  \sum_j a_{2^j}^2\Big( \sum_{R\in \Pi_j} n_{j,R}^{\sigma/p} \Big)^{2/\sigma} \Big]^{1/2},
\end{equation}
where we used that (by definition) $n_{j,R} \neq 0$ iff $R\in\Pi_j$. Now we consider
the inner sum over $R\in \Pi_j$ for fixed $j$ and write
\[
    \sum_{R\in \Pi_j} n_{j,R}^{\sigma/p} = \card\Pi_j \sum_{R\in\Pi_j}  \frac{n_{j,R}^{\sigma/p}}{\card\Pi_j}
    \leq (\card\Pi_j)^{1-\sigma/p} \Big( \sum_{R\in\Pi_j} n_{j,R} \Big)^{\sigma/p} 
\]
by Jensen's inequality.
By Lemma~\ref{lem.bernst} again we have $\sum_{R\in\Pi_j} n_{j,R} \leq 
\nu 2^j$ and $\card\Pi_j \leq \nu 2^j$,
implying
\[
    \sum_{R\in \Pi_j} n_{j,R}^{\sigma/p} \lesssim 2^j.
\]
Inserting this in \eqref{eq:E}, we get
\[
    \Big( \sum_{R\in \Pi} E_p(f,R)^\sigma \Big)^{1/\sigma}  \lesssim \Big[ \sum_j 2^{2j/\sigma} a_{2^j}^2  \Big]^{1/2} = \Big[ \sum_j \big(2^{j(\alpha + 1/p)} a_{2^j}\big)^2 \Big]^{1/2}
 \simeq   \ \|f\|_{\mathcal A_2^\alpha},
\]
where the latter inequality is a consequence of Fact~\ref{intro.f3}. 
Since this inequality is true for every partition $\Pi$ into disjoint atoms or 
rings, the space $\app_2^\alpha$ embeds continuously into $V_{\sigma,p}$.
\end{proof}
\subsection{The embedding of $V_{\sigma,p}$ into $\mathcal A_\infty^\alpha$}
\label{sec:embed2}
In the same way as the embeddings of approximation spaces in Theorem~\ref{bernst.equiv}
depend on the Bernstein inequality $\operatorname{BI}$, the embedding of $V_{\sigma,p}$ in
$\mathcal A_\infty^\alpha$ depends on $\operatorname{BI}$.

\begin{theo}\label{thm:embed_var_app}
    Assume that $1<p<\infty$,  $0<\tau<p$ and 
      $\beta = 1/\tau-1/p$. Then,
    \begin{enumerate}[(A)]
        \item If the Bernstein inequality $\operatorname{BI}(\mathscr A,S,p,\tau)$ is not satisfied then 
        for all $\alpha > \beta$ and $\sigma$ given by $\alpha = 1/\sigma - 1/p$, the 
        variation space $V_{\sigma,p}$ is not continuously embedded in $\mathcal A_\infty^\alpha$.
        \item If the Bernstein inequality  $\operatorname{BI}(\mathscr A,S,p,\tau)$ is satisfied, 
        we have for all $0<\alpha<\beta$ and $\sigma$ given by $\alpha = 1/\sigma - 1/p$ the 
        continuous embedding of $V_{\sigma,p}$ into $\mathcal A_\infty^\alpha$.
    \end{enumerate}
\end{theo}
\begin{proof}
    \textsc{Proof of (A):} Suppose that $V_{\sigma,p}$ embeds continuously in $\mathcal A_\infty^\alpha$.
    Then, for all $0<\theta < 1$ and $0<q\leq\infty$, the interpolation space 
    $(L^p(S), V_{\sigma,p})_{\theta,q}$ embeds continously in $(L^p(S), \mathcal A_\infty^\alpha)_{\theta,q}$.
    By Theorem~\ref{theo.101}, the former space coincides with the approximation space 
    $\mathcal A_q^{\theta \alpha}(L^p(S),\mathcal C)$ corresponding to the dictionary $\mathcal C$
    of functions piecewise in $S$.  Moreover, by Fact~\ref{intro.f1}, the latter space coincides with 
    $\mathcal A_q^{\theta\alpha} = \mathcal A_q^{\theta\alpha}(L^p(S),\Phi)$. Together, this gives 
    \begin{equation}\label{eq:app_coincide}
     \mathcal A_q^{\theta \alpha}(L^p(S),\mathcal C) =   \mathcal A_q^{\theta\alpha}(L^p(S),\Phi).
    \end{equation}
By choosing $\theta<1$ so that 
    $\theta \alpha>\beta$ and employing Theorem~\ref{bernst.equiv}(A), we conclude that 
    $\operatorname{BI}(\mathscr A,S,p,\tau)$ is satisfied, finishing the proof of (A).

    \textsc{Proof of (B):}
    Let $K(f,t) = K(f,t, L^p(S), \app_\tau^\beta) = \inf \{ \|f-g\|_p + t\|g\|_{\app_\tau^\beta}
    : g\in \app_\tau^\beta \}$
    be the K-functional between the spaces $L^p(S)$ and
    $\app_\tau^\beta$. Let $\sigma_n(f) = \sigma_n(f, \cc)= \inf \{ \| f - g\|_p : g\in\Sigma_n(S)\}$.
By $\operatorname{BI}(\mathscr A,S,p,\tau)$, the general result \cite[Chapter 7, Theorem 5.1]{constr.approx}
implies
    \begin{equation}\label{eq:Kmu}
        K(f,2^{-j\beta}) \lesssim 2^{-j\beta} \Big(\|f\|_p^\mu +  \sum_{\ell=0}^j (2^{\ell\beta} \sigma_{2^\ell}(f))^\mu \Big)^{1/\mu}
    \end{equation}
    with $\mu = \min(1,\tau)$. 
    Now we take $g\in V_{\sigma,p}$. By the Jackson inequality Proposition~\ref{th.kfun.c} (i),
    we have for such functions that
    \[
        \sigma_{2^j}(g) \lesssim 2^{-j\alpha} |g|_{V_{\sigma,p}}.
    \]
    Insert this inequality in \eqref{eq:Kmu} to obtain
    \[
        K(g,2^{-j\beta}) \lesssim 2^{-j\beta} \Big( \|g\|_p^\mu + \sum_{\ell=0}^j (2^{\ell(\beta-\alpha)} |g|_{V_{\sigma,p}})^\mu \Big)^{1/\mu}
        \lesssim 2^{-j\alpha}\|g\|_{V_{\sigma,p}}.
    \]
 By Fact~\ref{thm:KA}, we have
    \[
        K(g,2^{-j\beta}) \simeq \| g-\mathscr G_{2^j} g\|_p + 2^{-j\beta} \| \mathscr G_{2^j} g\|_{\app_\tau^\beta},
    \]
    implying that
    \begin{equation}\label{eq:jacksonVsigma}
        \|g-\mathscr G_{2^j} g\|_p + 2^{-j\beta} \| \mathscr G_{2^j} g\|_{\app_\tau^\beta} \lesssim 2^{-j\alpha} \|g\|_{V_{\sigma,p}}.
    \end{equation}
    Denoting by $(a_n)$ the decreasing rearrangement of $(c_n)$ in the expansion $g = \sum_n c_n\varphi_n$,
    and using \eqref{eq:special_beta},
    we have $2^{j/\tau} a_{2^j} \leq \|\mathscr G_{2^j} g\|_{\app_\tau^\beta}$.
    This yields the inequality
    \[
        a_{2^j} \leq 2^{j(\beta-\alpha-1/\tau) } \|g\|_{V_{\sigma,p}} = 2^{-j/\sigma} \|g\|_{V_{\sigma,p}}.
    \]
    This gives
    \[
        \|g\|_{\mathcal A_{\infty}^\alpha} \lesssim \|g\|_{V_{\sigma,p}},
    \]
    finishing also the proof of (B).
\end{proof}

\subsection{The $K$-functional for the pair $(L^p(S), V_{\sigma,p})$ and greedy approximation. }\label{sec:k-functional}
The embedding results of Theorems~\ref{thm:embed_app_var} and \ref{thm:embed_var_app}
imply that the K-functional between $L^p(S)$ and $V_{\sigma,p}$ is attained 
for greedy approximation in the following weak sense:

\begin{theo}\label{thm:K_equiv}
    Let $1<p<\infty$, $0<\tau<p$.
    Assume that condition $\operatorname{BI}(\mathscr A,S,p,\tau)$ 
    is satisfied and set $\beta = 1/\tau - 1/p$.
    For all $0 < \gamma < \alpha<\beta$, $0 < \kappa \leq \infty$, and $0<\rho\leq \infty$, the following are 
    equivalent:
    \begin{enumerate}
        \item $\{ 2^{n\gamma} K(f, 2^{-n\alpha}, L^p(S), \app_{\rho}^\alpha) : n\geq 0\} \in \ell^\kappa$,
        \item $\{ 2^{n\gamma} ( \| f - \mathscr G_{2^n} f\|_p + 2^{-n\alpha} \|\mathscr G_{2^n}f\|_{\app_\rho^\alpha})  : n\geq 0\} \in \ell^\kappa$,
        \item $\{ 2^{n\gamma} K(f, 2^{-n\alpha}, L^p(S), V_{\sigma,p}) : n\geq 0\} \in \ell^\kappa$,
        \item $ \{ 2^{n\gamma} ( \| f - \mathscr G_{2^n} f\|_p + 2^{-n\alpha} \|\mathscr G_{2^n} f\|_{V_{\sigma,p}}) : n\geq 0\} \in \ell^\kappa$.
    \end{enumerate}
\end{theo}
\begin{proof}
We know for $0 < \kappa\leq \infty$ that by \eqref{eq:KA},
\[
K(f,n^{-\alpha}, L^p(S), \app_\kappa^\alpha) \simeq \| f - \mathscr G_n f\|_X + n^{-\alpha} \|\mathscr G_n f\|_{\app_\kappa^\alpha}.
\]
This gives that by Fact~\ref{intro.f1} and \eqref{eq:norm_int_space}, for all $0<\gamma < \alpha$ 
and all $0<\kappa\leq \infty$ the assertions (1) and (2) are both equivalent to $f\in \app_\kappa^\gamma$
for all $0<\rho\leq \infty$.
Since condition $\operatorname{BI}(\mathscr A,S,p,\tau)$ 
is satisfied, the embedding results Theorems~\ref{thm:embed_app_var}
and \ref{thm:embed_var_app} yield the inequalities 
\begin{equation}\label{eq:equiv_app_var}
   \| f - \mathscr G_n f\|_p + \frac{1}{n^\alpha} \| \mathscr G_n f\|_{\app_\infty^\alpha}
   \lesssim  \| f - \mathscr G_n f\|_p + \frac{1}{n^\alpha} \| \mathscr G_n f\|_{V_{\sigma,p}}
   \lesssim \| f - \mathscr G_n f\|_p + \frac{1}{n^\alpha} \| \mathscr G_n f\|_{\app_q^\alpha}
\end{equation}
for $q = \min(\sigma,2)$.
Thus, looking at the equivalences (1) and (2) above for the different choices of $\rho = q$
and $\rho = \infty$ gives that (4) is equivalent to (2).
Finally, combining Fact~\ref{intro.f1}, \eqref{eq:norm_int_space} and Theorem~\ref{theo.101},
we get the remaining equivalence of (1) and (3). Here we also used the fact that since 
$\operatorname{BI}(\mathscr A, S,p,\tau)$ is satisfied, the approximation spaces 
$\app_\rho^\alpha$ with respect to the two dictionaries $\cc$ and $\Phi$ coincide by 
Theorem~\ref{bernst.equiv}(B).
\end{proof}

\subsection{The system $\Phi$ need not be unconditional in $V_{\sigma,p}$} 
For the sake of the completeness, it would be desirable to know, or to have an example,
  in the following direction:  if  the Bernstein inequality $\operatorname{BI}(\mathscr A,S,p,\sigma)$
   is not satisfied, then
the greedy projection with respect to the $p$-normalized system $\Phi$ need
 not be  bounded on $V_{\sigma,p}$, or in other words, that 
condition (3) of Lemma \ref{lem:abstract} need not be satisfied for $\spp = L^p(S)$ and 
$\spy = V_{\sigma,p}$. In this direction, we consider a stronger property of the following form: 
 for a finite  subset $\Lambda \subset \Phi $, let   $P_\Lambda f = \sum_{\phi \in \Lambda} 
 \langle f, \phi \rangle \phi$. We give an example that if 
the  Bernstein inequality  $\operatorname{BI}(\mathscr A,S,p,\sigma)$  is not satisfied, 
then the projections $\{P_\Lambda, \Lambda \subset \Phi\} $ need not be uniformly bounded on  
 $V_{\sigma,p}$. 
 For some special choices of atoms, this is also a counterexample to 
 the boundedness of greedy approximation in $V_{\sigma,p}$ (see Remark~\ref{rem:counter_greedy}).
Note that $\operatorname{w2}^*(\mathscr A,S,p,\sigma)$ formulated in Definition \ref{def:w2star}  is 
equivalent to the Bernstein inequality $\operatorname{BI}(\mathscr A,S,p,\sigma)$ by 
the main result of \cite{part2}.
For simplicity, we consider $\Omega = [0,1)$, the Lebesgue measure $\lambda = |\cdot |$ and
the space $S = \operatorname{span}\{ \charfun_{\Omega}\}$.  
Recall  that in case of $S = \operatorname{span}\{ \charfun_{\Omega}\}$,
 the inequality defining condition  $\operatorname{w2}^*(\mathscr A,S,p,\sigma)$
  takes the form of inequality \eqref{eq:example}.
The aforementioned condition similar to $\operatorname{w2^*}$ then reads as follows:
\begin{defn}    
We say that condition $\operatorname{w3}(\mathscr A,p,\sigma)$ is satisfied
    if there exists $\rho \in (0,1)$ and $M>0$ such that for each 
    chain $X_0 \supset X_1\cdots \supset X_n$ of atoms with $|X_n| \geq \rho |X_0|$
    and for all  partitions
    $(\Lambda_i)_{i}$ of $\{0,\ldots,n-1\}$ into sets $\Lambda_i$ consisting of
    consecutive indices $\Lambda_i = \{ \lambda_i,\ldots,\lambda_{i+1}-1\}$ and
    for all $\Gamma_i \subset \Lambda_i$ such that with
    \[
        R_i = \bigcup_{j\in\Lambda_i} X_j\setminus X_{j+1} = X_{\lambda_i} \setminus X_{\lambda_{i+1}},\qquad Q_i = \bigcup_{j\in\Gamma_i} X_j\setminus X_{j+1},
    \]
    we have the inequality
    \begin{equation}\label{eq:w3_cond}
        \Big(\sum_{i} \min(|Q_i|,|R_i\setminus Q_i|)^{\sigma/p}\Big)^{1/\sigma} \leq M |X_0\setminus X_n|^{1/p}.
    \end{equation}
\end{defn}
    We note that $\operatorname{w2}^*(\mathscr A,S,p,\sigma)$ implies $\operatorname{w3}(\mathscr A,p,\sigma)$. Indeed, we have
    \[
        \min(|Q_i|, |R_i\setminus Q_i|)^{\sigma/p} \leq \Big(\sum_{j\in\Lambda_i} |X_j\setminus X_{j+1}|     \Big)^{\sigma/p}
        \leq \sum_{j\in\Lambda_i} |X_j\setminus X_{j+1}|^{\sigma/p}.
    \]
    We also note that the condition (not $\operatorname{w3}$) is not empty.
    To this end, we consider $\Omega = [0,1)$ equipped with the Lebesgue measure $\lambda = |\cdot |$.
    We define, for $\ell\in \{0,\ldots,2n\}$ the $\sigma$-algebra $\mathscr F_\ell$ to be generated 
    by the atoms $\{ [j / (4n), (j+1)/(4n)) :  0\leq j\leq \ell-1 \}$ and 
    $X_\ell := [\ell/(4n),1)$. Therefore, $(X_\ell)_{\ell=0}^{2n}$ is a 
    chain of atoms satisfying $X_0 = [0,1)$ and $X_{2n} = [1/2,1)$.
    Choosing $\Lambda_i = \{ 2i, 2i+1\}$ and $\Gamma_i = \{2i\}$ for $i=0,\ldots,n-1$, the left 
    hand side of \eqref{eq:w3_cond} is $n^{1/\sigma - 1/p}$, and therefore unbounded in $n$, whereas 
    $|X_0 \setminus X_{2n}| = 1/2$.

\begin{prop}\label{prop:notw3} Assume that condition $\operatorname{w3}(\mathscr A,p,\sigma)$ is 
    not satisfied for some parameters $1<p<\infty$ and $0<\sigma<p$.
    Then, for each positive integer $N$, there exists a chain of atoms $(X_i)_{i=0}^n$
    such that the function $f = \charfun_{X_0\setminus X_n} - \gamma \charfun_{X_n}$ 
    (with $\gamma$ chosen so that $\mathbb Ef=0$)
    satisfies
    \begin{enumerate}
        \item $\|f\|_{V_{\sigma,p}} = \big\| \sum_{\ell} \langle f,\varphi_\ell\rangle \varphi_\ell \big\|_{V_{\sigma,p}} \lesssim |X_0\setminus X_{n}|^{1/p}$,
        \item
              there exists a subset  $\Gamma$ of the Haar support of $f$ such that
              \begin{equation}\label{eq:non_uncond}
                  \Big| \sum_{\ell\in\Gamma} \langle f,\varphi_\ell\rangle \varphi_{\ell} \Big|_{V_{\sigma,p}} \gtrsim N |X_0\setminus X_{n}|^{1/p}.
              \end{equation}
    \end{enumerate}
\end{prop}
\begin{proof}
We  assume that
(not $\operatorname{w3}$) is satisfied and take, for arbitrary $N\in\mathbb N$
a chain of atoms $X_0 \supset X_1 \supset \cdots \supset X_n$ with $|X_n|\geq \rho|X_0|$
for  $\rho\geq 2/3$  
and a partition $(R_i)$ and  corresponding sets $(Q_i)$ such that
\begin{equation}\label{eq:not_w3}
    \Big(\sum_{i} \min(|Q_i|,|R_i\setminus Q_i|)^{\sigma/p}\Big)^{1/\sigma} \geq  N |X_0\setminus X_n|^{1/p}.
\end{equation}
We let $f = \charfun_{X_0\setminus X_{n}} - \gamma \charfun_{X_{n}} = \charfun_{X_0}- (1+\gamma)\charfun_{X_{n}}$
with $\gamma = |X_0\setminus X_{n}| / |X_{n}|\leq 1/2$ chosen such that $\mathbb E f = 0$.
Let $h_j$ for $j=0,\ldots,n-1$ be the Haar function normalized to the value $1$ on
$X_j\setminus X_{j+1}$ with support
$X_j$. Then we have the expansion 
 $f = \sum_{j=0}^{n-1} a_j h_j$ for some coefficients $(a_j)_{j=0}^{n-1}$.
It can be seen easily that those coefficients satisfy the estimates
\begin{equation}
\label{eq:sum_aj}
\begin{aligned}
    1&\leq a_j\leq 1+\gamma,\qquad  j\in \{0,\ldots,n-1\}, \\
    -\gamma &\leq \sum_{j=0}^\ell a_jh_j\leq 0 \text{ on } X_{\ell+1},\qquad \ell\in \{0,\ldots,n-1\}.
\end{aligned}
\end{equation}

    First, consider statement (1). 
    Since $f = 
        \charfun_{X_0\setminus X_{n}} - \gamma \charfun_{X_{n}} = \charfun_{X_0} - (1+\gamma)\charfun_{X_{n}}$
        is the difference of two functions in the dictionary $\cc$,
    by the Bernstein inequality of Proposition~\ref{th.kfun.c},
    \begin{align*}
        |f|_{V_{\sigma,p}} & \lesssim \|f\|_p \leq \|\charfun_{X_0\setminus X_{n}}\|_p
        +\gamma \|\charfun_{X_{n}}\|_p
        = |X_0\setminus X_{n}|^{1/p} + \frac{|X_0\setminus X_{n}|}{|X_{n}|} |X_{n}|^{1/p}                                                                             \\
                           & = |X_0\setminus X_{n}|^{1/p}\Big( 1 + \Big(\frac{|X_0\setminus X_{n}|}{|X_{n}|} \Big)^{1/p'} \Big)  \lesssim |X_0\setminus X_{n}|^{1/p}.
    \end{align*}
    This proves item (1).

    Next, we prove item (2).
    We take the index sets $\Lambda_i, \Gamma_i$ corresponding to
    the partition $R_i$ and the sets $Q_i$ from condition \eqref{eq:not_w3}.
    We consider the blocks $b_i = \sum_{j\in\Lambda_i} a_j h_j$ and their subblocks
    $g_i = \sum_{j\in \Gamma_i} a_j h_j$. Let $g := \sum_i g_i$ and $\Gamma=\cup_i \Gamma_i$.
    We want to estimate the variation
    norm $|g|_{V_{\sigma,p}}$ from below. In order to do that, we choose the
    special partition of  disjoint rings
    $R_i$, which yields
    \begin{equation}\label{eq:estvar}
        \Big|\sum_i g_i \Big|_{V_{\sigma,p}}^\sigma
        \geq \sum_{j} E_p\Big( \sum_{i} g_i, R_j\Big)^\sigma.
    \end{equation}
    Observe that for $i > j$, $\operatorname{supp} g_i \subseteq X_{\lambda_i}$ is disjoint
    to $R_j = X_{\lambda_j} \setminus X_{\lambda_{j+1}}$.
    Moreover, for $i<j$, the function $g_i$ is constant on $R_j = X_{\lambda_j}\setminus X_{\lambda_{j+1}}$.
    This implies
    \[
        E_p \Big(\sum_{i} g_i ,R_j\Big)  = E_p(g_j,R_j).
    \]
    We estimate $g_i(t)$ for $t\in Q_i$ from below we choose $j_0 \in \Gamma_i$ such that
    $t\in X_{j_0}\setminus X_{j_0+1}$ and write
    \[
        g_i(t) = a_{j_0}h_{j_0}(t) + \sum_{j\in \Gamma_i\setminus \{j_0\} } a_j h_j(t)
        \geq 1 - \Big | \sum_{j\in\Gamma_i\setminus \{j_0\}} a_j h_j(t)\Big| \geq 1/2,
    \]
    by equation \eqref{eq:sum_aj}.
    On the other hand, on $R_j\setminus Q_j$ we obtain, again by \eqref{eq:sum_aj},
    \[
        g_i \leq 0.
    \]
    Therefore,
    \[
        E_p(g_j, R_j) = \inf_{c\in\mathbb R} \| g_j - c\|_{L^p(R_j)}
        \gtrsim \min( |Q_j|^{1/p}, |R_j\setminus Q_j|^{1/p}),
    \]
    and, inserting this in \eqref{eq:estvar},
    \[
        \Big| \sum_{j\in\Gamma} a_jh_j \Big|_{V_{\sigma,p}}^\sigma =
        \Big|\sum_i g_i \Big|_{V_{\sigma,p}}^\sigma        \gtrsim
        \sum_j\min(|Q_j|,|R_j\setminus Q_j|)^{\sigma/p}.
    \]
    Thus, it remains to use assumption \eqref{eq:not_w3} to deduce item (2).
\end{proof}
\begin{rem}\label{rem:counter_greedy}
    If we are able to choose the atoms freely, the expression on the left 
    hand side of \eqref{eq:non_uncond} can be a greedy approximant of $f$.
    Indeed, a slight variation of the example above showing that condition $\operatorname{w3}$ 
    is not empty, proves this fact as follows.
    We choose $X_0=[0,1)$ and $X_{2n}=[1/2,1)$ but the atoms $X_\ell$ for $1\leq \ell\leq 2n-1$
    are chosen such that with $S_j = X_j\setminus X_{j+1}$,
    $t :=|S_{2j}|$ is the same for each $j=0,\ldots, n-1$
    and also $s := |S_{2j+1}|$ is the same for each $j =0,\ldots,n-1$ 
    with the restriction that $t^{1/p} = 2s^{1/p}$. In the notation of the above proof, we note that 
    \[
        \| h_j \|_p = \big( |S_j| + |S_j|^p |X_{j+1}|^{1-p}\big)^{1/p}.
    \]
    We can choose $n$ arbitrarily large, in particular, we choose $n$ such that 
    \[
        |S_j|^{1/p} \leq \|h_j\|_p   \leq 2 |S_j|^{1/p}.
    \]
    In combination with the estimates \eqref{eq:sum_aj} for the coefficients $a_j$ this shows 
    that 
    \[
        \max_j (a_{2j+1} \|h_{2j+1}\|_p)  \leq \min_j (a_{2j} \|h_{2j}\|_p),
    \]
    and therefore, $\mathscr G_n f = \sum_{\ell \in\Gamma} \langle f,\varphi_\ell\rangle \varphi_\ell$.
    Thus, the result of Proposition~\ref{prop:notw3} now shows that for this special choice 
    of atoms, the greedy approximant $\mathscr G_n$ is not uniformly bounded in $n$ on $V_{\sigma,p}$.
\end{rem}

\appendix

\section{Jackson and Bernstein inequalities in variation spaces \label{app.s1}}

In this section we present the proofs of results formulated in Section \ref{ssq.1}. 
As we have already mentioned, the results from section \ref{ssq.1} are generalizations
 of the corresponding results from \cite{hky.2000}. Their proofs are analogous to the 
 proofs of their counterparts in \cite{hky.2000}. However, we have decided to present 
 the details of these proofs for the sake of the completeness.

\subsection{Proof of Proposition \ref{th.kfun.c}. \label{app.0.ss1}}

The proof of the  Jackson inequality in Proposition \ref{th.kfun.c} 
(i) relies on Theorem 6.1 of  \cite{cdpx.1999}. The original statement in  \cite{cdpx.1999} 
deals with families of dyadic squares in $[0,1]^2$.
However, one can check that the same proof works in the setting as in 
Section \ref{sec:loc}, with the constant $2\nu$ for 
$\nu$-ary filtrations $(\mathscr F_n)$.
Let us recall the statement of Theorem 6.1 of  \cite{cdpx.1999} in this form:

\begin{theo}[{\cite[Theorem 6.1]{cdpx.1999}}]
\label{app.0.theo.cdpx}
Let $\phi: \cf \to [0,\infty)$ be a super-additive set function, i.e. 
for  $K_1, K_2 \in \cf $ with $K_1 \cap K_2 = \emptyset$, there is
$$
\phi(K_1) + \phi(K_2) \leq \phi(K_1 \cup K_2).
$$
Moreover, assume that 
$$
\lim_{\mu \to \infty} \sup_{K \in \caa_\mu} \phi(K) =0.
$$
Fix $\varepsilon > 0$ such that $\phi(\Omega) > \varepsilon$.

Then there exist 
two families $\cpp_\varepsilon$ and $\tilde\cpp_\varepsilon$, each consisting  of pairwise disjoint atoms or rings, such that
\begin{itemize}
\item[(a)] $\cpp_\varepsilon$ is a partition of $\Omega$, consisting of pairwise disjoint atoms and 
rings, and $\phi(K) \leq \varepsilon$ for each $K \in \cpp_\varepsilon$,
\item[(b)] $\tilde\cpp_\varepsilon$ consists of pairwise disjoint atoms and rings,
and $\phi(K) > \varepsilon$ for each $K \in \tilde\cpp_\varepsilon$,
\item[(c)] $\card \cpp_\varepsilon \leq 2\nu \cdot \card \tilde \cpp_\varepsilon $.
\end{itemize}
\end{theo}

Now, Jackson inequality in Proposition \ref{th.kfun.c} (i) is proved 
 similarly as 
Theorems 6.2 and 6.3 in  \cite{cdpx.1999}, or Jackson inequality
in Theorem  6 (i) in  \cite{hky.2000}.

\begin{proof}[Proof  Proposition \ref{th.kfun.c} (i)]
Fix  $f\in V_{\sigma,p}(S)$. Without loss of generality, we can assume that $f \not \in S$, so $| f |_{V_{\sigma,p} }> 0$.  
Fix a positive integer $m$ and take $\varepsilon = \frac{1}{m^{p/\sigma}} | f |_{V_{\sigma,p}}^p$.  Consider the set function 
 \begin{equation}
 \phi(K) = E_p(f,K)^p.
 \end{equation}
 This function satisfies the assumptions of   Theorem \ref{app.0.theo.cdpx} by Proposition~\ref{prop.Wmu.f.est}.
Let $\cpp_\varepsilon$ and $\tilde\cpp_\varepsilon$ be the implied families of pairwise disjoint atoms and rings.
With the choice of $\varepsilon = \frac{1}{m^{p/\sigma}} | f |_{V_{\sigma,p}}^p$, we use  Theorem \ref{app.0.theo.cdpx} (b) to get
$\card \tilde\cpp_\varepsilon \leq m$, by the following argument:
$$\frac{\card \tilde\cpp_\varepsilon}{m} | f |_{V_{\sigma,p}}^\sigma
= \card \tilde\cpp_\varepsilon\cdot \varepsilon^{\sigma/p}
< \sum_{K \in \tilde\cpp_\varepsilon } E_p(f,K)^\sigma 
\leq | f |_{V_{\sigma,p}}^\sigma.
$$

Denote $\xi(m) = \card \cpp_\varepsilon$. Note that by Theorem \ref{app.0.theo.cdpx} (c)
we have $\xi(m) \leq 2\nu m$.  We combine this observation 
with Theorem \ref{app.0.theo.cdpx} (a)  to get
$$ 
\sigma_{\xi (m)} (f, \cc)^p
 \lesssim  \sum_{K\in \cpp_\varepsilon} E_p(f,K)^p \leq 
2\nu \frac{1}{m^{p \beta}} | f |_{V_{\sigma,p}}^p.
$$
This is enough to get Jackson inequaliy in  Proposition \ref{th.kfun.c} (i).
\end{proof}

Using also Lemma \ref{lem.bernst}, we are ready to give a proof of the Bernstein inequality in Proposition \ref{th.kfun.c} (ii).
The proof is similar to the proof of the Bernstein inequality  in Theorem 6 (ii) in \cite{hky.2000}

\begin{proof}[Proof of  Proposition \ref{th.kfun.c} (ii).]
 Let $g\in\Sigma_n^\cc$ and  $\Pi$ be a family  of pairwise disjoint atoms or rings. 
 We select the subfamily
\[
\Pi(g) = \{ R\in\Pi : E_p(g,R)>0 \}.
\]
Because $g\in\Sigma_n^\cc$, it follows by   Lemma \ref{lem.bernst}  that $\card\Pi(g) \leq n$. Therefore, by H\"older's inequality
\begin{align*}
\Big( \sum_{R\in\Pi} E_p(g,R)^\sigma \Big)^{1/\sigma} & =   \Big( \sum_{R\in\Pi(g)} E_p(g,R)^\sigma \Big)^{1/\sigma}
 \leq  n^\beta 
 \Big( \sum_{R\in\Pi(g)} E_p(g,R)^p \Big)^{1/p}  
\leq n^\beta \| g \|_p.
\end{align*}
Taking supremum over all families of disjoint atoms and rings $\Pi$, we get
\[
|g|_{V_{\sigma,p}} \leq  n^\beta  \| g \|_p.
\]
This is completes the proof of the Bernstein inequaliy in  Proposition \ref{th.kfun.c} (ii).
\end{proof}

For the sake of completeness, we present the proof of Lemma \ref{lem.bernst}.

\begin{proof}[Proof of Lemma \ref{lem.bernst}.] Assume that there exists  an atom $Q\in{\mathcal P}$ such that $E_p(f,Q)\neq 0$. Let us consider different positions of $I$ relative to $Q$:
\begin{itemize}
\item $I\cap Q=\emptyset$; this is not possible because then $f=0$ on $Q$;
\item $Q\subseteq I$; this is also not possible, because $f\cdot \charfun_Q \in S_Q$ and then $E_p(f,Q)=0$;
\item so the only possibility is $I\subsetneq Q$; but then for any $R\in{\mathcal P}$, $R\neq Q$
\[
R\cap Q=\emptyset \Rightarrow I\cap R = \emptyset.
\]
Thus, for all $R\in{\mathcal P}$, $R\neq Q$ we get $E_p(f,R) = 0$.
\end{itemize}

Assume now that $R=A\setminus B\in{\mathcal P}$, $\emptyset \neq B\subsetneq A$ are atoms,
 $R$ is not an atom and  $E_p(f,R)\neq 0$.  Let us now consider different positions of $I$ relative to $A$ and $B$.
\begin{itemize}
\item $A\cap I = \emptyset$: this is not possible because $f=0$ on $A\supset R$.
\item $A\subseteq I\Rightarrow R\subset I$ meaning that $f\charfun_R\in S_R$ what implies $E_p(f,R)=0$;
\item $I\subsetneq A$. Now we are going to review different positions of $I$ relative to $B$:
\begin{itemize}
\item $I\cap B=\emptyset$, which implies $I\subset R$. Because the sets belonging to ${\mathcal P}$ are pairwise disjoint there must be $I\cap P=\emptyset$ (and consequently $E_p(f,P)=0$) for $P\neq R$;
\item if $I\subset B$, then $I\cap R=\emptyset$ and $E_p(f,R)=0$;
\item what remains is the case of $B\subsetneq I \subsetneq A$. We are going to show
 that for any ring $P$ and atom $Q$ disjoint to $R$ there is $E_p(f,P)=0$ or 
 $E_p(f,Q)=0$.

First let $Q$ be an atom disjoint to  $R$.  Then we have a few cases:
\begin{itemize}
\item if $Q\cap A=\emptyset$ then surely $E_p(f,Q)=0$;
\item $A\subset Q$ is not possible because then $R\subset Q$;
\item $Q\subset A$ implies $Q\subset B$ (because otherwise $Q\cap R\neq \emptyset$). 
Then also $Q\subset I$ and $f\charfun_Q\in S_Q$ implying that $E_p(f,Q)=0$.
\end{itemize}
Now let $P=C\setminus D$ be a ring disjoint with $R=A\setminus B$, $P$ not an atom.  Let us review the position of $P$ relative to $R$
\begin{itemize}
\item $C\cap A=\emptyset$ implies that $f=0$ on $C\supset P$ and consequently $E_p(f,P)=0$;
\item $A\subseteq C$: in this case we need to review different positions of $A,B$ and $D$. 

 The case $A\cap D=\emptyset$ is not possible because it leads to $R\subset A\subset P$. 

In case $A \subseteq D$, there is   $I \cap P = \emptyset$ and $E_p(f,P) =0$.

It remains to consider the case $D\subsetneq A$.
\begin{enumerate}
\item If $D\cap B=\emptyset$  then either $A=B\cup D$, meaning that $R=D$, so $R$ is an atom, which we have excluded, or $A=B\cup D\cup Z$ for certain $\emptyset\neq Z\subset R$. But then also $Z\subset P$, meaning that $P\cap R\neq \emptyset$, which is not possible.
\item If $B\subset D$, then $A\setminus D\subset P$, $A\setminus D\subset R$ and $A\setminus D\neq\emptyset$ --  but this  is not possible, because of 
$R \cap P = \emptyset$.
\item $D\subset B$ implies that $R\subset P$ which is not possible,
because of 
$R \cap P = \emptyset$.
\end{enumerate}
\item $C\subsetneq A$. Because $R\cap P=\emptyset$ there must be $C\subset B$. 
In this case we have $P\subset C\subset B\subset I$, so $f\charfun_P\in S_P$ implying that
 $E_p(f,P)=0$.
\end{itemize}
\end{itemize}
\end{itemize}
This completes the proof of the lemma.
\end{proof}

\subsection{Proof of Theorem~\ref{theo.234}}
\label{sec:appendix_K_equiv}

Recall that 
the modulus of smoothness $\cw_S$ is given by  
\begin{equation*}
\cw_S(f,t)_{\sigma,p} = \sup_{\Pi} \min\Big(t,\frac{1}{\card\Pi}\Big)^\beta \left( \sum_{R\in\Pi} E_p(f,R)^\sigma\right)^{1/\sigma},
\end{equation*}
where $\beta = 1/\sigma - 1/p$ and the supremum is taken over all finite partitions $\Pi$ of $\Omega$ into disjoint atoms or rings.
In order to prove Theorem~\ref{theo.234},
we begin with some simple observation,  see Fact \ref{fact.kfun.a} below. Fact \ref{fact.kfun.a} (i)  corresponds to  \cite[(2.1)]{hky.2000},
while Fact \ref{fact.kfun.a} (ii) corresponds to the inequality \cite[(2.2)]{hky.2000}.
\begin{fact}\label{fact.kfun.a} (i) Let $f\in L^p$. Then
\begin{equation}
\cw_S(f,t)_{\sigma,p} \leq \Vert f\Vert_p.
\end{equation}
(ii)  There exists a constant $C$ such that for all 
$g\in V_{\sigma,p}$ and $0 < t \leq 1$, the following the inequality holds:
\begin{equation}
\cw_S(g,t)_{\sigma,p} \leq t^\beta |g|_{V_{\sigma,p}}.
\end{equation}
\end{fact}

\begin{proof}
To prove part (i), let us fix a finite partition $\Pi$ of $\Omega$ into disjoint atoms or rings.
Then, by H\"older's inequality
\begin{eqnarray*}
\sum_{R\in\Pi} E_p(f,R)^\sigma & \leq & \sum_{R\in\Pi}\Vert f\Vert_{L^p(R)}^\sigma 
\\ & \leq &
\Big( \sum_{R\in\Pi} 1^{p/(p-\sigma)} \Big)^{(p-\sigma) / p}  \Big(\sum_{R\in\Pi}\Vert f\Vert_{L^p(R)}^{\sigma\cdot p/\sigma}\Big)^{\sigma / p} \leq 
(\card\Pi)^{1-\sigma / p} \Vert f\Vert_{L^p(\Omega)}^\sigma,
\end{eqnarray*}
and consequently
\[
\Big(\sum_{R\in\Pi} E_p(f,R)^\sigma \Big)^{1/\sigma} \leq (\card\Pi)^{\beta} \Vert f\Vert_{L^p(\Omega)}.
\]
It follows that
\[
\min\Big(t,\frac{1}{ \card\Pi}\Big)^\beta\Big(\sum_{R\in\Pi} E_p(f,R)^\sigma \Big)^{1/\sigma} \leq \frac{1}{ (\card\Pi)^{\beta}}\Big(\sum_{R\in\Pi} E_p(f,R)^\sigma \Big)^{1/\sigma} \leq \Vert f\Vert_{L^p(\Omega)}.
\]
Taking the supremum over all finite partitions $\Pi$ of $\Omega$
into disjoint atoms or rings  completes the proof of part (i).

Next, let us proceed with part (ii). For this,  take any finite partition $\Pi$ of $\Omega$
into disjoint atoms or rings. Then we have
\[
\min\Big(t,\frac{1}{\card\Pi}\Big)^\beta \Big(\sum_{R\in \Pi} E_p(g,R)^\sigma\Big)^{1/\sigma} \leq t^\beta  \Big(\sum_{R\in \Pi} E_p(g,R)^\sigma\Big)^{1/\sigma} \leq t^\beta|g|_{V_{\sigma,p}}.
\]
Taking supremum over all $\Pi$ completes the proof of part (ii). 
\end{proof}

The proof of Theorem \ref{theo.234} relies on the following Jackson and Bernstein inequalities.
The Jackson inequality formulated as   Fact \ref{fact.kfun.b} (i)
is a generalization of the Jackson inequality as in 
\cite[Theorem 6 (i)]{hky.2000}.
The Bernstein inequality formulated as   Fact \ref{fact.kfun.b} (ii)
is a generalization of the Bernstein inequality as in
\cite[Theorem 6 (ii)]{hky.2000}.
\begin{fact}  \label{fact.kfun.b} 
(i)  [Jackson inequality] There exists a constant $C$ such that for all $f \in L^p(S)$
\begin{equation*} 
\sigma_n(f,\cc) \leq C \cw_S(f,1/n)_{\sigma,p}.
\end{equation*}

(ii)  [Bernstein inequality]   Assume $g\in\Sigma_n^{\cc}$. Then
\begin{equation*}
|g|_{V_{\sigma,p}} \leq C n^\beta \cw_S(g,\frac{1}{n})_{\sigma,p}.
\end{equation*}
\end{fact}

We proceed with the proof of Jackson inequality as stated in 
Fact \ref{fact.kfun.b} (i). It is analogous  to  the proof of the version of 
Jackson inequality  stated in 
Proposition \ref{th.kfun.c} (i). In particular, it also relies on 
Theorem~\ref{app.0.theo.cdpx}.

\begin{proof}[Proof of Fact \ref{fact.kfun.b} (i)]
Fix  $f\in L^p(S)$ and $\varepsilon = \frac{1}{n} \cw_S(f,\frac{1}{n})_{\sigma,p}^p$.  
Consider the set function 
 \begin{equation}
 \phi(K) = E_p(f,K)^p,
 \end{equation}
 which satisfies the assumptions of  \ref{app.0.theo.cdpx}
 by Proposition~\ref{prop.Wmu.f.est}.
Without loss of generality, assume that $E_p(f,\Omega)^p > \varepsilon$.
By Theorem~\ref{app.0.theo.cdpx}, we choose two finite families 
$\cpp_\varepsilon = \cpp_\varepsilon(f)$ and 
$\tilde \cpp _\varepsilon = \tilde \cpp _\varepsilon (f)$, both consisting 
of pairwise disjoint 
 atoms or rings, such that:
\begin{itemize}
\item[(a)] $\cpp_\varepsilon$ is a covering of $\Omega$,
and
 $
  E_p(f,R)^p \leq \varepsilon
 $
 for each $R\in \cpp_\varepsilon$.
\item[(b)]  
 $ \card\cpp_\varepsilon \leq 2\nu\card \tilde \cpp_\varepsilon$ and 
$
 E_p(f,R)^p > \varepsilon
 $
 for each $R\in \tilde \cpp_\varepsilon$. 

\end{itemize}
We shall check that  
\begin{equation}
\label{app.2.eq.45}
\card\cpp_\varepsilon \leq 2\nu \cdot n.
\end{equation}
Indeed, if \eqref{app.2.eq.45} is true, we have because of estimate (a)
 \begin{eqnarray*} 
\sigma_{2\nu \cdot n} (f, \cc)^p
& \lesssim & \sum_{R\in \cpp_\varepsilon} E_p(f,R)^p \leq \card\cpp_\varepsilon \cdot \varepsilon
\\ & \leq & 2\nu \cdot  \cw_S\big(f,\frac{1}{n}\big)_{\sigma,p}^p
\leq (2\nu)^{1+ \beta \cdot p}  \cw_S\big(f,\frac{1}{\xi \cdot n}\big)_{\sigma,p}^p.
 \end{eqnarray*}
This is enough to get Fact \ref{fact.kfun.b} (i).

It remains to check  \eqref{app.2.eq.45}. 
 For this, suppose on the contrary that $\card\cpp_\varepsilon > 2\nu n$.
  Denote $N_\varepsilon = \card\tilde \cpp_\varepsilon$. Because of (b), it follows that 
$
2\nu n < \cpp_\varepsilon \leq 2\nu N_\varepsilon,
$
and consequently  
\begin{equation}
\label{app.2.eq.89} n < N_\varepsilon.
\end{equation}
Now, using \eqref{app.2.eq.89} and  (b) we find
\[
N_\varepsilon \varepsilon^{\sigma / p} \leq \sum_{R\in \tilde \cpp_\varepsilon} E_p(f,R)^\sigma = 
 (N_\varepsilon)^{\sigma\beta}   \min\big(\frac{1}{n},\frac{1}{N_\varepsilon}\big)^{\sigma\beta}  \sum_{R\in \tilde \cpp_\varepsilon} E_p(f,R)^\sigma \leq  N_\varepsilon^{\sigma\beta}  \cw_S\big(f,\frac{1}{n}\big)_{\sigma,p}^\sigma,
\]
and therefore
\[
N_\varepsilon^{1/p} \varepsilon^{1/p} \leq    \cw_S\big(f,\frac{1}{n}\big)_{\sigma,p}.
\]
Using the choice $\varepsilon = \frac{1}{n} \cw_S(f,\frac{1}{n})_{\sigma,p}^p$ we get
\[
N_\varepsilon^{1/p}  \frac{1}{n^{1/p}} \cw_S\big(f,\frac{1}{n}\big)_{\sigma,p} \leq    \cw_S\big(f,\frac{1}{n}\big)_{\sigma,p},
\]
which implies
\[
N_\varepsilon \leq n.
\]
But this contradicts \eqref{app.2.eq.89}.  This contradiction shows that 
\eqref{app.2.eq.45} is true.
This completes the proof of item (i) of Fact \ref{fact.kfun.b}.
\end{proof}

Now, we proceed with the proof of Bernstein inequality as stated in 
Fact \ref{fact.kfun.b} (ii). It is analogous  to  the proof of the version of 
Bernstein inequality  stated in 
Proposition \ref{th.kfun.c} (ii). In particular, it also relies on Lemma \ref{lem.bernst}.

\begin{proof}[Proof of Fact \ref{fact.kfun.b} (ii).] The proof is similar to the proof of  \cite[(6.2)]{hky.2000}. 
    Let $g\in\Sigma_n^\cc$ and $\Pi$ an arbitrary finite family of pairwise disjoint atoms or rings. 
We set 
\[
\Pi(g) = \{ R\in\Pi : E_p(g,R)>0 \}.
\]
Because $g\in\Sigma_n^\cc$ we obtain from Lemma \ref{lem.bernst} that $\card\Pi(g) \leq n$. 
Therefore
\[
\Big( \sum_{R\in\Pi} E_p(g,R)^\sigma \Big)^{1/\sigma} = \Big( \sum_{R\in\Pi(g)} E_p(g,R)^\sigma \Big)^{1/\sigma} \leq n^\beta  \cw_S(g,\frac{1}{n})_{\sigma,p}.
\]
Taking supremum over all partitions $\Pi$ we get
\[
|g|_{V_{\sigma,p}} \leq  n^\beta  \cw_S(g,\frac{1}{n})_{\sigma,p},
\]
completing the proof.
\end{proof}

Now we are ready to give the proof of Theorem~\ref{theo.234}.
It follows the lines of the proof of  \cite[Theorem 5]{hky.2000}. 

\begin{proof}[Proof of Theorem \ref{theo.234}] 
First,  take $g\in V_{\sigma,p}$. Then,  by Fact \ref{fact.kfun.a} (i) and (ii), we have
\[
\cw_S(f,t)_{\sigma,p} \lesssim  \cw_S(f-g,t)_{\sigma,p} + \cw_S(g,t)_{\sigma,p} \leq   \Vert f-g\Vert_p + t^\beta |g|_{V_{\sigma,p}}.
\]
Taking the infimum over all $g\in V_{\sigma,p}$ proves the left inequality.

Let us proceed with the right inequality. Take $g\in\Sigma_n^\cc$. 
It follows by  Fact  \ref{fact.kfun.a}  (ii) and  \ref{fact.kfun.b}  (ii) that 
\[
\frac{1}{n^\beta} |g|_{V_{\sigma,p}} \simeq  \cw_S\big(g,\frac{1}{n}\big)_{\sigma,p}.
\]
So we have
\begin{align*}
\frac{1}{n^\beta} |g|_{V_{\sigma,p}}&\lesssim  \cw_S\big(g,\frac{1}{n}\big)_{\sigma,p} \lesssim 
 \cw_S\big(f,\frac{1}{n}\big)_{\sigma,p} + \cw_S\big(g-f,\frac{1}{n}\big)_{\sigma,p} 
\\
&\lesssim
\cw_S\big(f,\frac{1}{n}\big)_{\sigma,p} + \Vert g-f\Vert_p,
\end{align*}
where the last inequality follows by Fact \ref{fact.kfun.a} (i). Hence
\[
\Vert f-g\Vert_p + \frac{1}{n^\beta} |g|_{V_{\sigma,p}} \lesssim  \Vert f-g\Vert_p + \cw_S\big(f,\frac{1}{n}\big)_{\sigma,p},
\]
and therefore, since by the Bernstein inequality in Proposition~\ref{th.kfun.c} we have $\Sigma_n^\cc \subset V_{\sigma,p}$,
\[
K\big(f,\frac{1}{n^\beta},L_p(S), V_{\sigma,p}\big) \leq \Vert f-g\Vert_p + \frac{1}{n^\beta}|g|_{V_{\sigma,p}} 
\lesssim  \| f-g\|_p + \cw_S\big(f,\frac{1}{n}\big)_{\sigma,p}.
\]
Taking the infimum over all $g \in \Sigma_n^\cc$ yields
$$
K\big(f,\frac{1}{n^\beta},L_p(S), V_{\sigma,p}\big) \lesssim
 \sigma_n(f,\cc) + \cw_S\big(f,\frac{1}{n}\big)_{\sigma,p}.
$$
By Fact~\ref{fact.kfun.b} (i) we conclude
\[
K\big(f,\frac{1}{n^\beta},L_p(S), V_{\sigma,p}\big) \lesssim \cw_S\big(f,\frac{1}{n}\big)_{\sigma,p},
\]
which finishes the proof of the right inequality.
\end{proof}

\section{Another example for comparing variation spaces}
\label{sec:tensor_example}
Here we give another example that shows that variation space $V_{\sigma,p}$ 
corresponding to binary filtrations and $\nu$-ary filtrations don't necessarily 
coincide. 
Consider the probability space $\Omega = [0,1)^2$, equipped with the Borel $\sigma$-algebra and 
the two-dimensional Lebesgue measure and the choice $S = \operatorname{span}\{ \charfun_\Omega \}$.
Let $n_1,n_2 \geq 2$ be integers and let $\mathscr C_n$ be 
the collection of all rectangles of the form 
\[
  [(j_1-1) n_1^{-n}, j_1 n_1^{-n}  )   \times  [(j_2-1) n_2^{-n}, j_2 n_2^{-n}),\qquad j=(j_1,j_2)
\] 
for integers $j_1 = 1,\ldots, n_1^n$ and $j_2=1,\ldots,n_2^n$. 
Let $V_n := [1-n_1^{-n}, 1) \times [0,1)$ be the vertical strip at the right boundary 
of $\Omega$ of width $n_1^{-n}$.
Let the filtration $(\mathscr F_n)$ on $\Omega$
be such that each  $\mathscr F_n$ is generated by the rectangles 
in $\mathscr C_n$.
Denote by
$\mathscr C_{n,r} := \{ A\in \mathscr C_n : A\subset V_n \} $
the collection of all rightmost atoms of $\mathscr C_n$. Next, let $\mathscr C_{n,\ell}$ 
be the collection of leftmost atoms of $\mathscr C_n$ contained in the strip $V_{n-1}$.
For each atom $A\in \mathscr C_{n-1}$ contained in the strip $V_{n-1}$, choose two arbitrary 
atoms $B_1,B_2\in \mathscr C_{n,\ell}$ with $B_1,B_2\subset A$.
Enumerate all those resulting atoms as $(M_{n,m})$ for $m = 1,\ldots, 2n_2^{n-1}$ such 
that $M_{n,m}$ and $M_{n,m+1}$ have the same predecessor in $\mathscr C_{n-1}$ for all
odd $m$. 
Observe that we have $|M_{n,m}| = (n_1 n_2)^{-n}$ for all $m$.
For the most instructive special case $n_1 = n_2 = 2$, the first few sets $(M_{j,m})$ are depicted in 
Figure~\ref{fig:counter1}.

\begin{figure}
\begin{tikzpicture}[x=18cm,y=9cm]
    \foreach \n in {2,4,8}
    \foreach \m in {1,...,\n}
    { 
        \draw[draw=black] (1-1/\n, \m/\n-1/\n) node[anchor=south west]
        {$M_{\pgfmathparse{int(ceil(log2(\n)))}\pgfmathresult,\m}$} rectangle (1-1/\n/2,\m/\n);
        \ifthenelse{\n = 8}
            {\draw[draw=black](1-1/\n/2, \m/\n-1/\n) node[anchor=south west]{$\cdots$} rectangle (1,\m/\n)}
            {};
    }
\end{tikzpicture}
\caption{The unit square $[0,1)^2$ and the respective positions of the squares $M_{j,m}$ in 
the case $n_1 = n_2 = 2$.}
\label{fig:counter1}
\end{figure}


From the filtration $(\mathscr F_n)$, we generate a binary filtration with corresponding 
atoms $\mathscr A$ by introducing $\sigma$-algebras between $\mathscr F_{n-1}$ and $\mathscr F_n$.
In doing so we can, and will, make sure that $M_{n,m} \cup M_{n,m+1}$ is an atom in $\mathscr A$.

We now consider the variation spaces $V_{\sigma,p}^{\mathscr A}$ and $V_{\sigma,p}^{\mathscr C}$
(with $\mathscr C = \cup_n \mathscr C_n$)
corresponding to those two choices of atoms. The aim is to construct  a function for which 
those two norms are not comparable. Let us remark that since every atom/ring in $\mathscr C$
is an atom/ring in $\mathscr A$, it is clear that 
\[
    |f|_{V_{\sigma,p}^{\mathscr C}} \leq |f|_{V_{\sigma,p}^{\mathscr A}}
\]
for every function $f$.
Thus, we want to give an example of a function $f$ so that $|f|_{V_{\sigma,p}^{\mathscr A}}$ 
is much larger than $|f|_{V_{\sigma,p}^{\mathscr C}}$.
Using the sets $(M_{n,m})$ from above, we define the function
\[
    f_j := \alpha_j    \Big( \sum_{m \text{ odd}} \charfun_{M_{j,m}} - 
     \sum_{m \text{ even}} \charfun_{M_{j,m}} \Big)
\]
with $\alpha_j = n_1^{j/p} n_2^{j(1/p-1/\sigma)}$. 
Additionally set 
\[
    f := \sum_{j= 1}^K f_j.
\]


Let us give a lower estimate for $| f |_{V_{\sigma,p}^{\mathscr A}}$ by using the \emph{atoms} 
$ M_{j,m} \cup M_{j,m+1}$ for odd $m$.
To wit,
\begin{align*}
    |f|_{V_{\sigma,p}^{\mathscr A}}^\sigma \geq \sum_{j= 1}^K \sum_{m\text{ odd}} E_p(f_j, M_{j,m} \cup M_{j,m+1})^\sigma
     = \sum_{j= 1}^K \sum_{m\text{ odd}} \|f_j \charfun_{M_{j,m} \cup M_{j,m+1}}\|_p^\sigma,
\end{align*}
where we used that $f_j$ has mean zero over the sets $M_{j,m}\cup M_{j,m+1}$.
We evaluate this explicitly by the definition of $f_j$ as follows 
\begin{align*}
\sum_{j=1}^K \sum_{m\text{ odd}} \|f_j \charfun_{M_{j,m} \cup M_{j,m+1}}\|_p^\sigma
 &=   \sum_{j= 1}^K \sum_{m\text{ odd}} \alpha_j^\sigma (2|M_{j,m}|)^{\sigma/p}
 \simeq  \sum_{j= 1}^K \sum_{m\text{ odd}} \alpha_j^\sigma (n_1 n_2)^{-j\sigma/p} \\
 & \simeq  \sum_{j= 1}^K  n_2^j ( n_1^{\sigma j/p} n_2^{\sigma j(1/p - 1/\sigma)}) (n_1 n_2)^{-j\sigma/p} = 
 \sum_{j= 1}^K 1 = K.
\end{align*}
This yields the lower estimate 
\begin{equation}\label{eq:nary_lower}
|f|_{V_{\sigma,p}^{\mathscr A}} \geq K^{1/\sigma}.
\end{equation}

Next we estimate $|f|_{V_{\sigma,p}^\mathscr C}$ from above. Assume that $\Pi$ is a 
collection of pairwise 
disjoint atoms/rings in $\mathscr C$. This implies in particular that every $R\in \Pi$
is of the form $R = A\setminus B$ where $A\in\mathscr C$ and $B\in\mathscr C$ or 
the empty set. 

Let $A = I\times J\in\mathscr D$ with $\sup I = 1$. Then the function $f$ has zero mean over $A$.
Define the subset $\Pi'\subset \Pi$ to consist of all $R = A\setminus B \in \Pi$ so that 
if $A = I \times J$ we have $\sup I = 1$. Then, if $R\in \Pi\setminus \Pi'$ we have $\sup I < 1$, 
which gives by definition that $f$ is constant on $A$ and, a fortiori, $f$ is constant on $R$.
This implies $E_p(f,R)=0$ in this case and yields 
\[
\sum_{R\in\Pi} E_p(f,R)^\sigma = \sum_{R\in \Pi'} E_p(f,R)^\sigma.
\]
Let $R=A\setminus B \in \Pi'$. For $A = I\times J$ we have $\sup I=1$ and thus 
there exists at least one maximal interval $L_R$ of the form $[(j-1)n_2^{-n}, j n_2^{-n})$ satisfying
%
 \[
[\inf I, 1] \times L_R \subseteq R.    
 \]
 If $B=\emptyset$ we have that $|L_R| = |J|$ and if $B\neq \emptyset$ we have 
 $|L_R| = |J|/n_2$.
Due to this property and the disjointness of the sets in $\Pi$, the collection $\{ L_R : R\in \Pi' \}$ consists
of mutually disjoint dyadic intervals.

We now calculate $E_p(f, R)$ for $R = A\setminus B \in \Pi'$ with $A\in\mathscr C_j$:
\begin{align*}
    E_p(f,A\setminus B)^p &\leq E_p(f,A)^p = \| f\|_{L^p(A)}^p \leq \sum_{m = j}^K \alpha_m^p n_2^{m-j} (n_1n_2)^{-m} \\
    &=\sum_{m = j}^K n_1^m n_2^{m(1-p/\sigma))} n_2^{m-j} (n_1n_2)^{-m}  \\
    &=  n_2^{-j} \sum_{m=j}^K n_2^{m(1-p/\sigma)}
    \simeq 
     n_2^{-j} n_2^{j(1-p/\sigma)} = n_2^{-jp/\sigma},
\end{align*}
since $p/\sigma > 1$. This implies 
\begin{equation}\label{eq:nary_upper}
E_p(f,R)^\sigma \lesssim |L_R|,\qquad R\in \Pi',
\end{equation}
where we used the fact that $|L_R|$ and the height of  $A$ are comparable.
 Therefore we see that 
 \[
    \sum_{R\in\Pi} E_p(f,R)^\sigma = 
    \sum_{R\in\Pi'} E_p(f,R)^\sigma
    \lesssim \sum_{R\in\Pi'} |L_R|\leq 1,
 \]
 by the disjointness of the intervals $\{ L_R : R\in\Pi'\}$.
 Since this calculation was done for arbitrary partitions $\Pi$,
this yields
\[
|f|_{V_{\sigma,p}^{\mathscr C}} \lesssim 1.
\]
Together with \eqref{eq:nary_lower}, this means that the two norms 
$|\cdot|_{V_{\sigma,p}^{\mathscr C}}$ and   $|\cdot|_{V_{\sigma,p}^{\mathscr A}}$
are not comparable.

\section{Non-separability of variation spaces}\label{app:sep}
This section contains a proof of Theorem~\ref{thm:not_sep}.
We consider an arbitrary probability space $(\Omega,\mathscr F)$ and a
finite dimensional space $S\subseteq L^\infty(\Omega)$ as usual and
a binary filtration $(\mathscr F_n)_{n=0}^\infty$ and corresponding atoms $\mathscr A$ that satisfy
\begin{equation}\label{eq:density}
  \text{each chain of atoms $(X_i)_{i=0}^L$ of length $L+1$ satisfies $|X_L|\leq |X_0|/2$.} 
\end{equation}
for some uniform constant $L$.
Recall the notation $A'$ for the smaller direct successor
of $A$ corresponding to the filtration $(\mathscr F_n)$ and 
$A''$ for the larger direct successor of $A$.
Every $A\in\mathscr A$ admits those because of 
the density assumption $\lim_{n\to\infty} \max_{A\in\mathscr A_n} |A| =0$.
For each atom $A\in\mathscr A$, we choose one function $\varphi_A$ that is part
of a local orthonormal system $\Phi=(\varphi_j)$ with $\operatorname{supp} \varphi_A = A$ 
corresponding to $(\mathscr F_n)$ and $S$ 
and $\varphi_A$ is orthogonal to the functions in $S$ restricted to $A$. 

\begin{lem}\label{lem:chain}
    Let $(X_i)_{i=0}^K$ for $K\leq L$ be a chain of atoms with $|X_K| \leq |X_0|/2$. 
    Then, there exists an index $i\in \{1,\ldots,K\}$ such that 
    \[
        |X_i| \leq c_K |X_{i-1}| \leq c_L |X_{i-1}|
    \]
    with $c_K = 2^{-1/K}$.
\end{lem}
\begin{proof}
Indeed, if $|X_{i}| > c_K |X_{i-1}|$ for each $i=1,\ldots, K$, we get 
\[
|X_K| > c_K^K |X_0| \geq 2 c_K^K |X_K|,    
\]
by assumption, which is a contradiciton. 
\end{proof}

\begin{lem}\label{lem:D} 
    If \eqref{eq:density} is satisfied and
     $A\in\mathscr A$ is an arbitrary atom, there exists a 
    subatom $D=D(A)\subset A$ satisfying 
    \[
        |D''| \leq c_L |D| \qquad \text{and}\qquad |D| \geq |A|/2.
    \]
    In particular, we also have $|D'| \geq (1-c_L) |D|$.
\end{lem}
\begin{proof}
Consider the chain $(X_i)_{i=0}^L$,
given by $X_0 = A$ and $X_i = X_{i-1}''$ for $i=1,\ldots,L$.
Let $K\leq L$ be the minimal index such that $|X_K| \leq |X_0|/2$. 
By Lemma~\ref{lem:chain}, we take an index $i\in \{1,\ldots,K\}$ satisfying 
$|X_i| \leq c_L |X_{i-1}|$. Taking $D=D(A) = X_{i-1}$
concludes the proof.
\end{proof}

Let $0< p<\infty$, $0<\sigma<p$ and $\beta = 1/\sigma - 1/p>0$.
Let $\alpha = \min(1,p)$.
We define a subsequence $(\mathscr G_j)_{j=0}^\infty$ of the filtration $(\mathscr F_n)$,
and a sequence of functions $(f_j)$, 
by the following procedure.
Let $\mathscr G_0 = \mathscr F_0 = \{\emptyset, \Omega\}$ and take 
$D = D(\Omega)$ according to Lemma~\ref{lem:D}. Set $f_0 = \psi_D$,
which is contained in $S(\mathscr F_{n_0})$
for some positive integer $n_0$.
Let $\mathscr G_{j-1}$ and $f_{j-1}\in S(\mathscr F_{n_{j-1}})$ be given. 
Let $\mathscr G_{j} = \mathscr F_{n_{j-1}}$ and for each atom $C$ 
of $\mathscr G_j$, choose $D(C)\subset C$ according to Lemma~\ref{lem:D}.
Let $\mathscr A_j^{\mathscr G}$ be the set of atoms of $\mathscr G_j$ 
and define 
\[
f_j = \sum_{C\in \mathscr A_j^{\mathscr G}} |C|^{1/2 + \beta} \varphi_{D(C)},
\]
which is contained in $S(\mathscr F_{n_j})$ for some positive integer $n_j > n_{j-1}$.

Having defined $f_j$ for each integer $j\geq 0$, we set for $\Lambda\subset \mathbb N$
\begin{equation}\label{eq:def_f}
    f_\Lambda := \sum_{j\in\Lambda} f_{\nu\cdot j} =    \sum_{j\in\Lambda}  \sum_{C\in\mathscr A_{\nu\cdot j}^{\mathscr G}} |C|^{1/2+\beta} \varphi_{D(C)},
\end{equation}
for some fixed $\nu\in\mathbb N$ for which we will give a condition later.
Let $\mathscr A^{\mathscr G} = \cup_j \mathscr A_j^{\mathscr G}$.
We summarize that for all $C\in \mathscr A^{\mathscr G}$, we have 
by Lemma~\ref{lem:D} the estimates
\begin{equation}\label{eq:summary}
    |D(C)'| \geq (1-c_L) |D(C)| \geq (1-c_L) |C|/2.
\end{equation}

\begin{lem}\label{lem:p_orth}
    Let $C \in \mathscr A^{\mathscr G}$. Let $h$ be a function of the 
    form $h = h_1 \charfun_{D(C)'} + h_2\charfun_{D(C)''}$ with $h_1,h_2\in S$.
    Then, for each $0<p<\infty$ we have
    \[
        \| h \|_p \simeq |D(C)|^{1/p - 1/2} \|h\|_2\simeq |C|^{1/p - 1/2} \|h\|_2,    
    \]
    where the implicit constants depends only on the constants $c_1,c_2$ from 
    \eqref{eq:L1Linfty}, on $c_L$ and on $p$.
\end{lem}
\begin{proof}
    Using \eqref{eq:L1Linfty} and \eqref{eq:summary}, we estimate 
    \[
        \|h \|_p \simeq |D(C)|^{1/p} \|h\|_{D(C)}.
    \]
    This immediately implies the conclusion.
\end{proof}

\begin{lem}\label{lem:f_j}
    Let $R\subset \Omega$ and  $\mathscr B\subset \mathscr A^{\mathscr G}$ be a finite disjoint collection  
    such that $D(C)\cap R \in \{\emptyset,D(C)\}$ 
    for each $C\in \mathscr B$.

    Then we have the following estimate:
    \[
        \Big\| \sum_{C\in\mathscr B} |C|^{1/2+\beta}\varphi_{D(C)} \Big\|_{L^p(R)}^p \simeq 
\sum_{C\in \mathscr B, D(C)\subseteq R} |C|^{p/\sigma}.
    \]

    Moreover,  if $R$ is a finite union of atoms of $\mathscr F_m$ and 
    $j_0$ is the smallest index such that 
    $\mathscr G_{j_0}\supseteq\mathscr F_m$, we have for every $\ell\geq 0$
    \[
       \Big\| \sum_{j\geq j_0+\ell} f_j\Big\|_{L^p(R)}^\alpha \leq \sum_{j\geq j_0+\ell} \| f_j \|_{L^p(R)}^\alpha \lesssim c_L^{\alpha\beta\ell}|R|^{\alpha/\sigma}.
    \]
\end{lem}
\begin{proof}
    Indeed, since  $\|\varphi_{D(C)}\|_2=1$, we get from Lemma~\ref{lem:p_orth}
\begin{equation}
    \label{eq:f_j}
\begin{aligned}
    \Big\| \sum_{C\in\mathscr B} |C|^{1/2 + \beta} \varphi_{D(C)}\Big\|_{L^p(R)}^p
    &= \sum_{C\in\mathscr B,D(C)\subseteq R } |C|^{p/2 + \beta p} \|\varphi_{D(C)}\|_p^p                        \\
                         & \simeq \sum_{C\in\mathscr B, D(C)\subseteq R} |C|^{p/2+\beta p} |C|^{1-p/2} 
                         = \sum_{C\in\mathscr B, D(C)\subseteq R} |C|^{p/\sigma}.
\end{aligned}
\end{equation}
This concludes the proof of the first part.

If the sets $C\in \mathscr C$ form an arbitrary partition of the set $R$ that satisfies
$|C| \leq \gamma^{k} |R|$ for each $C\in\mathscr C$ and some $\gamma < 1$,
we introduce the splitting
\[
    \mathscr C_\ell = \{ C\in\mathscr C : \gamma^{\ell+1} |R| < |C| \leq \gamma^\ell |R| \}
\]
with $\mathscr C_\ell = \emptyset$ for $\ell<k$.
Moreover,  $\operatorname{card} \mathscr C_\ell \leq \gamma^{-(\ell+1)}$.
Using those things we obtain
\begin{equation}\label{eq:geom}
\begin{aligned}
    \sum_{C\in\mathscr C} |C|^{p/\sigma} &= \sum_{\ell=k}^\infty \sum_{C\in\mathscr C_\ell} |C|^{p/\sigma}
    \leq |R|^{p/\sigma} \sum_{\ell=k}^\infty \gamma^{\ell p/\sigma} \cdot \operatorname{card} \mathscr C_\ell       \\
                                         & \lesssim
    |R|^{p/\sigma} \sum_{\ell=k}^\infty \gamma^{\ell(p/\sigma-1)} \lesssim |R|^{p/\sigma} \cdot \gamma^{k(p/\sigma-1)}.
\end{aligned}
\end{equation}

For $C\in \mathscr A_j^{\mathscr G}$ with $C\subset R$ we have by construction of 
the filtration $(\mathscr G_j)$ that $|C| \leq c_L^{j-j_0} |R|$.
We use the above splitting and \eqref{eq:f_j} for the setting  
$\mathscr B = \mathscr A_j^{\mathscr G}$ for $j\geq j_0$
and $\gamma = c_L$ to get 
    \[
        \sum_{j\geq j_0+\ell} \| f_j \|_{L^p(R)}^\alpha \lesssim c_L^{\ell\beta\alpha} |R|^{\alpha/\sigma},
    \]
    concluding the proof of the lemma.
\end{proof}

In particular, the above lemma implies $f_\Lambda\in L^p$ for every $\Lambda \subset \mathbb N$.

\begin{prop}\label{prop:uncount_set}
    Let $\nu$ in \eqref{eq:def_f} be a sufficiently large constant.

    Then, for every $\Lambda\subset \mathbb N$, 
    the function $f_\Lambda$ defined in \eqref{eq:def_f} satisfies
    \begin{enumerate}
        \item $f_\Lambda \in V_{\sigma,p}$,
        \item for every $\Gamma\subset \mathbb N$ with $\Gamma \neq \Lambda$
         we have $| f_\Lambda - f_\Gamma|_{V_{\sigma,p}} \gtrsim 1$, 
         where the implicit constant is independent of $\Lambda$ or $\Gamma$.
    \end{enumerate}
\end{prop}

\begin{proof}
Let $R = A\setminus B$ for two atoms $A,B$ of the filtration $(\mathscr F_n)$
($B=\emptyset$ is possible). 
Choose $n$ such that $A\in\mathscr A_n$ and choose $m$ such that 
$B\in \mathscr A_m$. Then, $R = A\setminus B$ is a finite union
of atoms of $\mathscr F_m$.
Let $j_0$ be the smallest index such that
$\mathscr G_{j_0\nu}\supseteq\mathscr F_n$.
We now estimate $E_p(f,R)$ from above.

\textsc{Case 1: $|R| \geq (1-c_L)|A|$. }
Here we have by Lemma~\ref{lem:f_j}
\begin{align*}
    E_p(f_\Lambda,R)^\alpha &\leq E_p(f_\Lambda,A)^\alpha = 
    \Big\| \sum_{j\geq j_0, j\in\Lambda} f_{j\nu}\Big\|_{L^p(A)}^\alpha \\
    &\leq \sum_{j\geq j_0,j\in\Lambda} \|f_{j\nu}\|_{L^p(A)}^\alpha
    \lesssim |A|^{1/\sigma}\lesssim |R|^{1/\sigma}.
\end{align*}
\textsc{Case 2: } $|R|<(1-c_L)|A|$. This means that $|B| > c_L|A|$.
We observe that $\mathscr G_{(j_0+1)\nu}\supsetneq\mathscr F_m$.
Indeed, if this would not be the case, a superset $D\supseteq B$ would be an atom 
of $\mathscr G_{(j_0+1)\nu}$ and a set $E$ with $D\subset E\subset A$ is an atom of $\mathscr G_{j_0\nu}$.
By construction of the filtration $(\mathscr G_j)$, we have $|D| \leq c_L |E|$, which
yields a contradiction as follows:
\[
|E| \leq |A| < c_L^{-1}|B| \leq c_L^{-1} |D| \leq |E|.
\]

Then, 
\begin{equation}\label{eq:G}
    \begin{aligned}
        E_p(f_\Lambda,R)^\alpha & =  E_p( f_\Lambda - P_n f_\Lambda, R)^\alpha
        = E_p\Big( \sum_{j\geq j_0,j\in\Lambda} f_{j\nu},R\Big)^\alpha \\
        &\leq \sum_{j > j_0,j\in\Lambda} \|f_{j\nu}\|_{L^p(R)}^\alpha + E_p(f_{j_0 \nu},R)^\alpha  \\
        &\lesssim |R|^{\alpha/\sigma} + E_p \Big( \sum_{C\in\mathscr A_{j_0\nu}^{\mathscr G}} |C|^{1/2+\beta} \varphi_{D(C)},R\Big)^\alpha,
    \end{aligned}
\end{equation}
where in the last step, we used Lemma~\ref{lem:f_j}.
Next, we observe that for every $C\in\mathscr A_{j_0\nu}^{\mathscr G}$, the 
strict inclusion $B\subsetneq D(C)$ is impossible.
Indeed, this would mean that $B\subseteq D(C)'$ or $B\subseteq D(C)''$
in the underlying case and  the construction of $D(C)$ gives the contradiction 
\[
    c_L |D(C)| \leq c_L |A| < |B| \leq |D(C)''| \leq c_L |D(C)|.
\]
Let $\mathscr C_A = \{ C\in \mathscr A_{j_0\nu}^{\mathscr G} : C\subseteq A\}$. Then
this argument and the nestedness of atoms shows
that for $C\in\mathscr C_A$, we have one of the two possibilities  
    $B\cap D(C) = \emptyset$ or $D(C)\subseteq B$. This means that 
    $\mathscr B = \mathscr C_A$ also satisfies the assumptions of Lemma~\ref{lem:f_j}, 
    and thus
\begin{align*}
    E_p \Big( \sum_{C\in\mathscr A_{j_0\nu}^{\mathscr G}} |C|^{1/2+\beta} \varphi_{D(C)},R\Big)^\alpha
    &= E_p \Big( \sum_{C\in\mathscr C_A } |C|^{1/2+\beta} \varphi_{D(C)},R\Big)^\alpha \\
    &\lesssim \Big\| \sum_{C\in\mathscr C_A} |C|^{1/2+\beta} \varphi_{D(C)}\Big\|_{L^p(R)}^\alpha \\
    &\simeq \Big(\sum_{C\in\mathscr C_A, D(C)\subseteq R} |C|^{p/\sigma}\Big)^{\alpha/p} \lesssim |R|^{\alpha/\sigma},
\end{align*}
since $p/\sigma>1$.
Combining cases 1 and 2, we have proved the inequality 
\[
E_p(f,R) \lesssim |R|^{1/\sigma}    
\]
for every atom or ring $R\subset \Omega$.
If now $\Pi$ is any partition 
of $\Omega$ into a finite number of disjoint atoms or rings, this estimate yields
\[
    \sum_{R\in\Pi} E_p(f,R)^\sigma \lesssim \sum_{R\in\Pi} |R| \leq 1.    
\]
Therefore we get $f\in V_{\sigma,p}$ and thus item (i) is proved.

Next, we want to prove for every $\Gamma\subset \mathbb N$ with $\Gamma\neq \Lambda$
 the inequality $|f_\Lambda-f_\Gamma|_{V_{\sigma,p}} \gtrsim 1$.
We choose $j_0$ to be the smallest integer contained in the 
symmetric difference $\Lambda\Delta\Gamma$ and set $\mathscr C := \mathscr A_{j_0}^{\mathscr G}$.
We use the collection $\mathscr C$ for estimating 
$|f_\Lambda-f_\Gamma|_{V_{\sigma,p}}$ as follows:
\begin{equation}\label{eq:lower1}
    |f_\Lambda - f_\Gamma|_{V_{\sigma,p}}^\sigma \geq \sum_{C\in\mathscr C} E_p(f_\Lambda - f_\Gamma , C)^\sigma
    = \sum_{C\in\mathscr C}  E_p\Big(\sum_{j\geq j_0} \varepsilon_j f_{j\nu},C\Big)^\sigma,
\end{equation}
where $\varepsilon_j = 1$ if $j\in\Lambda\setminus \Gamma$, $\varepsilon_j = -1$ if $j\in\Gamma\setminus \Lambda$
and $\varepsilon_j = 0$ otherwise.
Next, we estimate for $C\in\mathscr C$
\begin{equation}\label{eq:lower_triangle_var}
E_p\Big(\sum_{j\geq j_0} \varepsilon_j f_{j\nu},C\Big)^\alpha \geq E_p(f_{j_0\nu},C)^\alpha - 
\Big\| \sum_{j>j_0} f_{j\nu}\Big\|_{L^p(C)}^\alpha.
\end{equation}
By Lemma~\ref{lem:f_j}, we have for $C\in\mathscr C$
\begin{equation}\label{eq:small}
 \Big\| \sum_{j>j_0} f_{j\nu}\Big\|_{L^p(C)}^\alpha \lesssim c_L^{\alpha\beta\nu} |C|^{\alpha/\sigma}.
\end{equation}
For the other term, use Lemma~\ref{lem:p_orth} to obtain
\begin{equation}
    \label{eq:large}
\begin{aligned}
E_p(f_{j_0\nu}, C) &= \inf_{g\in S} \| |C|^{1/2+\beta}\varphi_{D(C)}  - g\|_{L^p(C)} \\
&\simeq| C |^{1/p-1/2} \inf_{g\in S}  \| |C|^{1/2+\beta}\varphi_{D(C)}  - g\|_{L^2(C)} \\
&\simeq | C |^{1/p-1/2}   \| |C|^{1/2+\beta}\varphi_{D(C)} \|_{L^2(C)} 
\simeq |C|^{1/\sigma},
\end{aligned}
\end{equation}
where we also used orthogonality of $\varphi_{D(C)}$ to functions contained 
in $S$ on $C$.
Using now \eqref{eq:small} and \eqref{eq:large}, we now choose $\nu$ 
sufficiently large such that 
\[
  \Big\| \sum_{j>j_0(C)} f_{j\nu}\Big\|_{L^p(C)}   \leq E_p(f_{j_0\nu},C) /2.
\]
Thus, inserting those things into \eqref{eq:lower1} and \eqref{eq:lower_triangle_var},
\begin{align*}
    |f_\Lambda-f_\Gamma|_{V_{\sigma,p}}^\sigma \gtrsim \sum_{C\in\mathscr C} |C| = 1.
\end{align*}
since $\mathscr C$ is a partition of $\Omega$.
This proves also item (2) and concludes the proof
of the proposition.
\end{proof}
We are now ready to give the 
\begin{proof}[Proof of Theorem~\ref{thm:not_sep}]
We just note that the proof of Theorem~\ref{thm:not_sep} is 
the same as the proof of Theorem~\ref{thm:sep} if we use Proposition~\ref{prop:uncount_set}
instead of Proposition~\ref{prop:counter2}.
\end{proof}

\section{Extensions to weak greedy algorithms}
\label{sec:appendix_quasi_greedy}
In this section we describe extensions of Lemma~\ref{lem:abstract} and Corollary~\ref{thm:KA}
to certain weak greedy algorithms.

First, we note that, with the same proof, Lemma~\ref{lem:abstract} has the following, more general variant.
\begin{lem}\label{lem:abstract_ext}
    Assume that $(\spp,\|\cdot \|)$ is a Banach space, $\Psi = (\psi_n)$ a greedy basis in $\spp$, and
    let $\mathcal Y\subset \spp$ a continuously embedded semi-quasi-normed space such that $\psi_n\in \mathcal Y$
    for every $n$. Let $\beta >0$ and let for $f\in \mathcal X$ the sequence $(A_n f)_{n\geq 1}$ be 
    a sequence of finite linear combinations of elements of $\Psi$ such that 
    \begin{enumerate}
        \item there exists a constant $m$ such that for all $f\in \mathcal X$ we have $A_n f\in \Sigma_{nm}^{\Psi}$,
        \item there  exists a constant $c>0$ such that for all $f\in \mathcal X$ we have 
        \[
            \| f - A_n f\| \leq c\cdot \sigma_n(f,\Psi).
        \]
    \end{enumerate}

    Then, the assertion of Lemma~\ref{lem:abstract} holds with $A_n f$ replacing 
    the greedy approximant $\mathscr G_n f$.
\end{lem}

We now discuss an example of a possible choice of $A_n f$ for $f\in \mathcal X$ satisfying conditions (1),(2) in the 
above lemma.
We let $\Psi = (\psi_n)_{n\geq 1}$ be a normalized greedy basis in $\mathcal X$ and we split the positive 
integers $\mathbb N$ into $\mathbb N = \cup_{j=1}^\infty \Gamma_j$ for some sets $\Gamma_j$ satisfying 
the following two conditions for some constant $m$:
\begin{enumerate}[(i)]
    \item $\Gamma_{j_1} \cap \Gamma_{j_2} = \emptyset$ for $j_1 \neq j_2$,
    \item $1\leq \card\Gamma_j \leq m$.
\end{enumerate}
We take the block of $\Psi$ corresponding to $\Gamma_j$ by setting
\[
W_j f = \sum_{n\in\Gamma_j}    \psi_n^*(f) \psi_n,
\]
denoting by $(\psi_n^*)$ the functionals dual to the basis $(\psi_n)$.
Additionally, we denote the largest coefficient of $f$ in $\Gamma_j$ by  
\[
\mu_j(f) = \max_{n\in\Gamma_j} |\psi_n^*(f)|.
\]
Then, we fix $0<t\leq 1$ and perform a weak greedy algorithm with respect to the sequence 
$(\mu_j(f))$, i.e., we choose a sequence $(j_k)_{k\geq 1}$ such that 
\begin{equation}\label{eq:def_mu}
    \min_{1\leq \ell \leq k} \mu_{j_\ell}(f) \geq t \max_{\ell \geq k+1} \mu_{j_\ell}(f).
\end{equation}
Using this sequence, put
\[
B_k f = \sum_{\ell = 1}^k W_{j_\ell} f.
\]
Without restriction, we assume that $B_k f$ has the property that for all positive integers $M,n$ we 
have 
\begin{equation}\label{eq:BM}
    B_n(B_M f) =\begin{cases}
        B_n f,& \text{if } n\leq M, \\
        B_M f,& \text{if } n > M.
    \end{cases}
\end{equation}
This is possible since each choice of $(j_k)$ satisfying \eqref{eq:def_mu} for $f\in \mathcal X$
also satisfies \eqref{eq:def_mu} for $B_M f$ replacing $f$.
\begin{prop}\label{prop:quasi_greedy}
    The choice $A_n=B_n$ satisfies conditions (1) and (2) of Lemma \ref{lem:abstract_ext}.
\end{prop}
\begin{proof}
    It is clear that condition (1) is satisfied with the upper bound $m$ for the cardinality 
    of $\Gamma_j$.

    In order to check condition (2), we show that there exists a constant $0<\theta<1$ such that 
    for each $f\in \mathcal X$ we have  
    a sequence $(n_i)_{i\geq 1}$ such that $\{ n_1,\ldots, n_k\} \subset \cup_{\ell=1}^k \Gamma_{j_\ell}$ 
    and 
    \[
        \min_{1\leq i\leq k} |\psi_{n_i}^*(f)| \geq \theta \max_{n\notin \{n_1,\ldots,n_k\}} |\psi_{n}^*(f)|.
    \]
    First, let us show that this is enough to show condition (2). To this end, let 
    \[
        L_k f = \sum_{i=1}^k \psi_{n_i}^*(f) \psi_{n_i},
    \]
    which is then a weak greedy algorithm with parameter $\theta$ and with respect 
    to the greedy basis $\Psi$. By \cite[Theorem 1.39]{greedy} we obtain 
    \[
        \| f - L_k f \| \leq C \sigma_k(f,\Psi)    
    \]
    for some constant $C$ depending on $\Psi$ and $\theta$. Since $\Psi$ is also an 
    unconditional basis in $\mathcal X$, we have 
    \[
        \| f - B_k f\| \lesssim \|f - L_k f\| \lesssim \sigma_k(f,\Psi).
    \]
    This implies that  condition (2) is satisfied with $A_k = B_k$.

    Now we proceed with the inductive construction of the sequence $(n_i)_{i\geq 1}$. We fix $0< s \leq 1$  and 
    put $\theta = ts$.
For $k\geq 1$,
    suppose we have already chosen $\{ n_1,\ldots,n_{k-1}\}$.     
    Then, take 
    $n_{k}\in Z_{k} := \cup_{\ell = 1}^{k} \Gamma_{j_\ell} \setminus \{n_1,\ldots,n_{k-1}\} $ 
    such that 
    \[
        |\psi_{n_{k}}^*(f)| \geq s \max_{n\in Z_{k}} |\psi_n^*(f)|
    \]
    and set $\tilde{Z}_{k} = Z_{k}\setminus \{n_{k}\}$.
    It is clear that this construction satisfies $\{n_1,\ldots,n_k\} \subset \cup_{\ell=1}^k \Gamma_{j_\ell}$.
    Therefore, the only thing left to check is that for each $k\geq 1$ we have the inequality 
    \begin{equation}\label{eq:n_k_left_to_prove}
        |\psi_{n_k}^*(f)| \geq \theta \max_{n\notin \{n_1,\ldots,n_k\}} |\psi_n^*(f)|.
    \end{equation}
    If $n$ is chosen such that $n \notin \{n_1,\ldots,n_k\}$ we have 
    \[
         n\in \tilde{Z}_k \subset \bigcup_{\ell=1}^k \Gamma_{j_\ell}
        \qquad\text{or}\qquad n \notin \bigcup_{\ell=1}^k \Gamma_{j_\ell}.
    \]
    If $n\in\tilde{Z}_k$ we have by definition 
    \[
        |\psi_{n_k}^*(f)| \geq s \max_{\ell\in Z_k} |\psi_\ell^*(f)| \geq s|\psi_n^*(f)| \geq \theta |\psi_n^*(f)|.
    \]
    On the other hand, if $n\notin \cup_{\ell=1}^k \Gamma_{j_\ell}$, we choose $\xi \notin \{ j_1,\ldots,j_k\}$ 
    such that $n\in \Gamma_\xi$. Then, since $\Gamma_{j_k} \subset Z_k$,
    \[
       |\psi_{n_k}^*(f) | \geq s \max_{\ell\in Z_k} |\psi_\ell^*(f)| \geq s\mu_k(f) \geq s t \mu_\xi(f)
       \geq st |\psi_n^*(f)| = \theta |\psi_n^*(f)|,
    \]
    finishing the proof of \eqref{eq:n_k_left_to_prove}. Therefore, the proof of the proposition is 
    complete.
\end{proof}

We now want to apply Lemma~\ref{lem:abstract_ext} to the setting of approximation spaces 
and get a result similar to Corollary~\ref{thm:KA} with $B_n f$ replacing $\mathscr G_n f$.
In order to do that, we now show that $\mathcal Y = \mathcal A_q^\beta(\mathcal X,\Psi)$
satisfies items (1)--(3) of Lemma~\ref{lem:abstract} with $B_k f$ instead of $\mathscr G_k f$.

Indeed, since item (2) of Lemma~\ref{lem:abstract} does not use $\mathscr G_n f$, the proof of Corollary~\ref{thm:KA} implies (2) of 
Lemma~\ref{lem:abstract}. In order to prove items (1) and (3) of Lemma~\ref{lem:abstract}, we will 
only consider the case $q<\infty$ (the case $q=\infty$ is proved by analogous arguments) and 
use that 
\begin{equation}\label{eq:approx_equiv}
    \| f\|_{\mathcal A_q^\beta} \simeq \|f\|_{\mathcal X} + \Big( \sum_{n=1}^\infty \big(n^\beta \sigma_n(f) \big)^q \frac{1}{n}\Big)^{1/q}.
\end{equation}
In order to see (1), we observe that Proposition~\ref{prop:quasi_greedy}
implies  $\| f - B_k f\| \lesssim \sigma_k(f)$. Therefore,
\[
k^\beta \|f - B_k f\| \lesssim k^\beta \sigma_k(f)     \lesssim 
\Big( \sum_{j=k/2}^k \big(j^\beta \sigma_j(f) \big)^q \frac{1}{j} \Big)^{1/q} \leq \|f\|_{\mathcal A_q^\beta},
\]
which implies (1) of Lemma~\ref{lem:abstract}.

In order to check (3), we first note that 
\[
\sigma_{mk}(f) \leq \|f - B_k f\|\lesssim \sigma_k(f),    
\]
and therefore, using \eqref{eq:approx_equiv},
\begin{equation}\label{eq:quasi_greedy}
    \| f\|_{\mathcal A_q^\beta} \simeq \|f\|_{\mathcal X} + \Big( \sum_{n=1}^\infty \big(n^\beta \| f - B_n f\|_{\mathcal X} \big)^q \frac{1}{n}\Big)^{1/q}.
\end{equation}
Let $(j_\ell)$ be a sequence such that \eqref{eq:def_mu} is satisfied for 
$f$. 
For arbitrary integers $M$, we have for each $n<M$ by the unconditionality of $\Psi$
\begin{align*}
    \| B_M f - B_n(B_M f) \| = \Big\| \sum_{\ell=n+1}^M W_{j_\ell} f \Big\|
    \lesssim \Big\| \sum_{\ell=n+1}^\infty W_{j_\ell} f\Big\| = \| f - B_n f\|,
\end{align*}
where we also used \eqref{eq:BM}. Using this inequality, the fact that $\Psi$ is an 
unconditional basis in $\mathcal X$, and the equivalence \eqref{eq:quasi_greedy},
we obtain
\begin{align*}
    \|B_m f\|_{\mathcal A_q^\beta} &\simeq \|B_M f\|_{\mathcal X} + \Big(\sum_{n=1}^\infty \big( n^\beta \| B_M f - B_n(B_M f)\|_{\mathcal X}\big)^q \frac{1}{n} \Big)^{1/q} \\
    &\lesssim \| f\|_{\mathcal X} + \Big(\sum_{n=1}^\infty \big( n^\beta \|  f - B_n f\|_{\mathcal X}\big)^q \frac{1}{n} \Big)^{1/q}
    \simeq \|f\|_{\mathcal A_q^\beta}.
\end{align*}
This also proves (3) of Lemma~\ref{lem:abstract}.

\subsection*{Acknowledgments}
	M. Passenbrunner is supported by the Austrian Science Fund FWF, projects P32342 and P34414.

\bibliographystyle{plain}
\bibliography{main-variation}

\end{document}